\documentclass[oneside,11pt]{amsart}
\usepackage[utf8]{inputenc}
\usepackage[T1]{fontenc}
\usepackage{amsmath, mathrsfs, bbm}
\usepackage{graphicx}
\usepackage[T1]{fontenc}
\usepackage{combelow}

\usepackage[small]{caption}

\usepackage{tikz-cd}
\usetikzlibrary{arrows,chains,matrix,positioning,scopes}

\makeatletter
\tikzset{join/.code=\tikzset{after node path={%
\ifx\tikzchainprevious\pgfutil@empty\else(\tikzchainprevious)%
edge[every join]#1(\tikzchaincurrent)\fi}}}
\makeatother
\tikzset{>=stealth',every on chain/.append style={join},
         every join/.style={->}}
\tikzstyle{labeled}=[execute at begin node=$\scriptstyle,
   execute at end node=$]

\usepackage{amsfonts}
\usepackage{amssymb}
\usepackage{amsxtra}

\usepackage[all]{xy}
\SelectTips{cm}{12}
\CompileMatrices

\usepackage{ifthen}

\newcommand{\showcomments}{yes}

\newsavebox{\commentbox}
{\ifthenelse{\equal{\showcomments}{yes}}%
{\footnotemark
    \begin{lrbox}{\commentbox}
    \begin{minipage}[t]{1.25in}\raggedright\sffamily\upshape\tiny
    \footnotemark[\arabic{footnote}]}
{\begin{lrbox}{\commentbox}}}%
{\ifthenelse{\equal{\showcomments}{yes}}%
{\end{minipage}\end{lrbox}\marginpar{\usebox{\commentbox}}}
{\end{lrbox}}}

\theoremstyle{plain}
\newtheorem{theorem}{Theorem}[section]
\newtheorem{corollary}[theorem]{Corollary}
\newtheorem{lemma}[theorem]{Lemma}
\newtheorem{proposition}[theorem]{Proposition}

\theoremstyle{definition}
\newtheorem{defn}[theorem]{Definition}

\newtheorem{definition}[theorem]{Definition}
\newtheorem{remark}[theorem]{Remark}

\newcommand{\bbp}{\mathbb{P}}
\newcommand{\bndry}{\partial}

\newcommand{\R}{\mathbb{R}}
\newcommand{\N}{\mathbb{N}}

\newcommand{\relbndry}{\bndry (G,\bbp)}

\newcommand{\boundary}{\partial}

\DeclareMathOperator{\Aut}{Aut}

\DeclareMathOperator{\Stab}{Stab}
\DeclareMathOperator{\Fix}{Fix}

\DeclareMathOperator{\Ends}{Ends}

\DeclareMathOperator{\Inc}{Inc}

\newcommand{\presentation}[2]{\langle\, {#1} \mid {#2} \,\rangle}

\newcommand{\bigset}[2]{ \bigl\{ \, {#1} \bigm| {#2} \, \bigr\} }

\renewcommand{\emptyset}{\varnothing}
\renewcommand{\setminus}{-}

\newcommand{\field}[1]{\mathbb{#1}}

\newcommand{\Q}{\field{Q}}

\renewcommand{\P}{\field{P}}

\newcommand{\of}{\circ}

\usepackage{combelow}

\newcommand{\Swiatkowski}{{\'{S}}wi{\k{a}}tkowski}
\newcommand{\Sierpinski}{Sierpi{\'n}ski}

\begin{document}
\title[Local connectedness of the Bowditch boundary]{Local connectedness of boundaries for relatively hyperbolic groups}

\author{Ashani Dasgupta}
\address{Faculty of Mathematics\\
Cheenta Academy\\
22 Lake Place Road\\
Kolkata 700029\\
India}
\email{ashani.dasgupta@cheenta.com}

\author{G. Christopher Hruska}
\address{Department of Mathematical Sciences\\
University of Wisconsin--Milwaukee\\
P.O. Box 413\\
Milwaukee, WI 53201\\
USA}
\email{chruska@uwm.edu}


\date{\today}

\begin{abstract}
Let $(\Gamma,\P)$ be a relatively hyperbolic group pair that is relatively one ended.
Then the Bowditch boundary of $(\Gamma,\P)$ is locally connected.
Bowditch previously established this conclusion under the additional assumption that all peripheral subgroups are finitely presented, either one or two ended, and contain no infinite torsion subgroups.
We remove these restrictions; we make no restriction on the cardinality of $\Gamma$ and no restriction on the peripheral subgroups $P \in \P$.
\end{abstract}

\maketitle

\section{Introduction}
\label{sec:Introduction}

One of the earliest major theorems proved about hyperbolic groups is the result that the Gromov boundary of a one-ended hyperbolic group is locally connected and without cut points.  The proof combines work of Bestvina--Mess, Levitt, Bowditch, and Swarup \cite{BestvinaMess91,Levitt98,Bowditch99_Treelike,Swarup96}.
Notable applications of this result include semistability (due to \mbox{Geoghegan} \cite{BestvinaMess91,GeogheganSwenson19}), the topological invariance of the JSJ decomposition \cite{Bowditch98_JSJ}, the one-dimensional boundary classification theorem of \cite{KapovichKleinier00}, and the existence of quasi-hyperbolic planes in hyperbolic groups \cite{BonkKleiner05}.

Subsequently Bowditch studied the local connectedness of boundaries of relatively hyperbolic groups. 
Bowditch introduces a natural boundary of a relatively hyperbolic pair, the Bowditch boundary, and establishes many of its fundamental properties in \cite{Bowditch12_RelHyp}.
In a series of articles \cite{Bowditch99_Boundaries,Bowditch99_Connectedness,Bowditch01_Peripheral}, he proves that the Bowditch boundary of a relatively hyperbolic group pair $(\Gamma,\P)$ is locally connected, if certain restrictions are imposed on the subgroups known as peripheral subgroups.  These restrictions require peripheral subgroups to be finitely presented, either one or two ended, and to contain no infinite torsion subgroups. In this paper we extend Bowditch's work and remove all restrictions, proving the following theorem. 

\begin{theorem}
\label{thm: intro main}
Let $(\Gamma,\P)$ be a relatively hyperbolic group pair with connected boundary $M=\boundary(\Gamma,\P)$.
Then $M$ is locally connected and every cut point of $M$ is the fixed point of a parabolic subgroup.

The boundary $M$ has a cut point with stabilizer $P \in \P$ if and only if $\Gamma$ splits relative to $\P$ over a subgroup of $P$.
\end{theorem}

A splitting as a graph of groups is said to be \emph{relative to $\mathbb{P}$} if every member of $\mathbb{P}$ is contained in a vertex group.

We note that this theorem is not restricted to finitely generated groups and imposes no restrictions on the peripheral subgroups.
Osin extends the notion of relative hyperbolicity to groups that need not be finitely generated in \cite{Osin06_RelHyp}.  
The construction and basic properties of the Bowditch boundary have been extended to this more general setting by Gerasimov and Potyagailo \cite{Gerasimov12,GerasimovPotyagailo15,GerasimovPotyagailo16}.
The boundary of a relatively hyperbolic group pair is always compact, and when $\Gamma$ is countable it is metrizable.
Thus, in the countable case, a connected Bowditch boundary is a Peano continuum.

We note that, throughout this paper, we consider only infinite peripheral subgroups.  (If any peripheral subgroup is finite, the boundary contains a dense set of isolated points and, hence, is not connected and also not locally connected.)
Under this convention, the boundary of $(\Gamma,\P)$ is connected if and only if $\Gamma$ does not split over a finite subgroup relative to $\P$.

In the case that the boundary of $(\Gamma,\P)$ is connected and contains cut points, the following result gives that the cut point tree in the sense of \cite{PapasogluSwenson06} is simplicial and provides a topological characterization of the boundary as the inverse limit of a tree of compacta in the sense of \Swiatkowski\ \cite{Swiatkowski20}.

\begin{theorem}[Cut point tree]
\label{thm: intro cut point}
Suppose $(\Gamma,\P)$ is relatively hyperbolic with connected boundary $M=\boundary(\Gamma,\P)$.
Consider the bipartite graph $T$ with vertex set $V_0\sqcup V_1$, where $V_0$ is the family of nontrivial maximal subcontinua of $M$ that are not separated by any cut point of $M$ and $V_1$ is the set of cut points of $M$. Vertices $v\in V_1$ and $B \in V_0$ are joined by an edge if $v \in B$. Then we have the following.
\begin{enumerate}
   \item The graph $T$ is a simplicial tree equal to the canonical JSJ tree of cylinders for splittings over parabolic subgroups relative to $\P$.
   \item Each group $\Stab(B)$ with $B\in V_0$ is finitely generated and relatively hyperbolic with boundary $B$, a Peano continuum with no cut points.
   \item The boundary $M$ is equal to the inverse limit of a tree of compacta, obtained by gluing the boundaries $B$ of the subgroups $\Stab(B)$ along points in the pattern of $T$.
\end{enumerate}
\end{theorem}

This characterization simplifies related treelike decompositions of boundaries described in the finitely generated case in \cite{Bowditch01_Peripheral,Dahmani03} and extends this treelike structure to groups that need not be finitely generated.
The canonical tree of cylinders is called the ``maximal peripheral splitting'' in \cite{Bowditch01_Peripheral}.
See Proposition~\ref{prop:JSJfinitelygenerated} for a description of this canonical tree.

Haulmark--Hruska have used Theorem~\ref{thm: intro main} in their proof that the JSJ decomposition of a one-ended relatively hyperbolic pair over elementary subgroups depends only on the topology of the boundary \cite{HaulmarkHruska_Canonical}.
Hruska--Walsh have used Theorem~\ref{thm: intro main} in their examination of relatively hyperbolic pairs with planar boundary \cite{HruskaWalsh_Planar}.
Regarding semistability at infinity, Theorem~\ref{thm: intro main}, together with work of Mihalik--Swenson and Haulmark--Mihalik, gives the following immediate result (see \cite{MihalikSwenson21,HaulmarkMihalik_RelativeSemistable} for more details).

\begin{corollary}
Suppose $(\Gamma,\P)$ is nonelementary relatively hyperbolic and $\Gamma$ is finitely generated and one ended.
\begin{enumerate}
    \item If each $P\in \P$ is one ended and $\Gamma$ does not split over any subgroup $Q\le P$ with $P \in \P$, then $\Gamma$ is semistable at infinity.
    \item If each $P \in \P$ is finitely presented and semistable at infinity, then $\Gamma$ is semistable at infinity.
\end{enumerate}
\end{corollary}

We also note that the main results of Haulmark \cite{Haulmark19} apply to any relatively hyperbolic pair with locally connected boundary.  (They are stated there with more restrictive hypotheses, but only local connectedness is used in the proofs.)
Thus we have the following immediate consequence.

\begin{corollary}
Let $(\Gamma,\P)$ be relatively hyperbolic.
Suppose $\Gamma$ does not split relative to $\P$ over an elementary subgroup and each $P \in \P$ is one ended.
If the boundary $\boundary(\Gamma,\P)$ is one dimensional, then it is homeomorphic to a circle, a \Sierpinski\ carpet, or a Menger curve.
\end{corollary}

We remark that de~Souza has independently obtained a variation of Bowditch's local connectivity theorem in \cite{deSouza_Blowups} using a method that is only valid when the relatively hyperbolic group $\Gamma$ is finitely presented and one ended.  Theorem~\ref{thm: intro main} includes and extends the result of de~Souza.

\subsection{Methods of proof}

The proofs of Theorems \ref{thm: intro main} and~\ref{thm: intro cut point} roughly follow the same outline as the prior proofs of Swarup and Bowditch.  However, several steps require new ideas to extend their conclusions from restricted settings to the broader setting of this paper.

In \cite{Bowditch12_RelHyp}, Bowditch constructs the boundary of a relatively hyperbolic pair $(\Gamma,\P)$ and proves that it is compact and Hausdorff, provided that $\Gamma$ is countable.  
In Section~\ref{sec:ConstructingBoundary}, we show that the same elementary construction extends to the general case with no cardinality restrictions on $\Gamma$.  The resulting Bowditch boundary has a simple, concrete description and coincides with the canonical boundary introduced by Gerasimov in \cite{Gerasimov12} as the abstract Cauchy--Samuel completion of a certain uniform space.
A reader who is only interested in countable relatively hyperbolic groups may safely skip this section, whose only purpose is to extend the reach of this paper beyond countable groups.

In Section~\ref{sec:TreesOfCompacta}, we show that a splitting over parabolic subgroups gives rise to a treelike decomposition of the boundary.  This observation is due to Bowditch \cite{Bowditch01_Peripheral} in the finitely generated case. Dahmani uses an elaborate hands-on construction of a tree of compacta to prove a combination theorem for relatively hyperbolic group pairs in \cite{Dahmani03}.
We simplify this construction by characterizing the treelike structure as the inverse limit of a tree system of compacta in the sense of \Swiatkowski\ \cite{Swiatkowski20}.
The inverse limit construction can be used in place of the elaborate construction of \cite{Dahmani03} and may lead to a deeper understanding of combination theorems.

This simplification allows us to eliminate all restrictions on the cardinality of $\Gamma$ and leads to characterizations of connectedness and local connectedness of boundaries in Section~\ref{sec:Connectedness}. 
Using an accessibility theorem for parabolic splittings (see Proposition~\ref{prop:JSJfinitelygenerated}), we reduce from the general case to the more well-behaved setting of finitely generated relatively hyperbolic groups.
After this point, all groups considered are finitely generated.

Section~\ref{sec:Pretrees} discusses pretrees and establishes results about non-nesting actions and full quotients of pretrees, which are used in the following two sections.
Section~\ref{sec:CutPointPretree} contains a brief summary of the cut point pretree construction of Swenson \cite{Swenson00} and some of its key dynamical properties.
We use this pretree in place of Bowditch's more technically elaborate ``pretree completion'' of \cite{Bowditch99_Treelike} to simplify the analysis in the following section.

Section~\ref{sec:  minimalcodense} is a study of a construction due to Bowditch that associates a dual dendrite $D$ to each continuum $M$ with a preferred family of cut points $C$.
Proposition~\ref{prop:MaximalHausdorff} characterizes this dendrite as the maximal Hausdorff quotient of $M$ in which any two points are identified if they are not separated by a point of $C$.
Bowditch \cite{Bowditch99_Treelike} establishes that this dendrite is nontrivial if $M$ admits a convergence action by a one-ended group $G$ containing no infinite torsion group and $M$ has a nonparabolic cut point.
The main goal of Section~\ref{sec:  minimalcodense} is to extend this result from one-ended groups to relatively one-ended groups.
This extension is one of the main innovations of the present proof; it makes essential use of the non-nesting property and the results about pretrees established in Sections \ref{sec:Pretrees} and~\ref{sec:CutPointPretree}.

In Section~\ref{sec:Isometric action of a relatively hyperbolic group} we verify that a theorem of Levitt \cite{Levitt98} about non-nesting actions on trees extends from finitely presented groups to groups with a finite relative presentation.
Section~\ref{sec:MainThm} contains an accessibility theorem for relatively hyperbolic groups and splittings over two-ended loxodromic subgroups.  It also contains the proofs of Theorems \ref{thm: intro main} and~\ref{thm: intro cut point}, which combine the ingredients discussed above.

\subsection{Acknowledgements}
We thank Arka Banerjee, Victor Gerasimov, Craig Guilbault, Aaron Messerla, Hoang Thanh Nguyen, Prayagdeep Parija, Kim Ruane, and the anonymous referee for helpful conversations, feedback, and corrections that influenced this paper.
The second author was partially supported by grant \#714338 from the Simons Foundation.

\section{Relatively hyperbolic groups and their splittings}
\label{sec: rel hyp groups}

\subsection{Relative hyperbolicity}

This section discusses relatively hyperbolic groups and their boundaries. 
For more details, see \cite{Bowditch12_RelHyp,GerasimovPotyagailo15,GerasimovPotyagailo16}.  Alternative approaches may be found in \cite{DrutuSapir05,Osin06_RelHyp,GrovesManning08}.

\begin{definition}
\label{def: graph, tree}
A \emph{graph} $K$ is a set $V=V(K)$ of \emph{vertices}, a set $E=E(K)$ of \emph{oriented edges}, an \emph{initial vertex} $\iota(\vec{e}\,)$, \emph{terminal vertex} $\tau(\vec{e}\,)$, and \emph{inverse edge} $-\vec{e}$ for each oriented edge $\vec{e}\in E(K)$ satisfying $-\vec{e}\ne\vec{e}$, $-(-\vec{e}\,)=\vec{e}$, and $\iota(\vec{e}\,) = \tau(-\vec{e}\,)$.
An (unoriented) \emph{edge} $e$ is a pair of oriented edges $\{\vec{e},-\vec{e}\,\}$.

In a graph $K$, a \emph{path} is indicated by giving its sequence of oriented edges or its sequence of vertices (if the graph is simplicial).
A path is \emph{simple} if its vertices are distinct.
A \emph{cycle} is a path from a vertex to itself, and a \emph{circuit} is a cycle with all vertices distinct except for the initial and terminal vertices, which are equal.
A \emph{simplicial tree} is a connected graph with no circuits.

We treat graphs as combinatorial objects, as above, or as topological or metric spaces depending on the context.
\end{definition}

\begin{definition}
A geodesic space $X$ is \emph{$\delta$--hyperbolic} if each side of a geodesic triangle lies in the $\delta$--neighborhood of the union of the other two sides.
\end{definition}

\begin{definition}
A \emph{group pair} $(\Gamma,\mathbb{P})$ consists of a group $\Gamma$ together with a finite collection $\mathbb{P}$ of subgroups of $\Gamma$.
Given a group pair $(\Gamma,\mathbb{P})$, an action of $\Gamma$ on $X$ is \emph{relative to $\mathbb{P}$} if each member of $\mathbb{P}$ has a fixed point in $X$.
\end{definition}

\begin{definition}[Relatively hyperbolic]
\label{Def: RelHyp}
A connected simplicial graph $K$ is \emph{fine} if for each $n$, each edge in $K$ is contained in only finitely many circuits of length $n$.
An action of a group pair $(\Gamma,\P)$ on a graph $K$ is \emph{relatively hyperbolic} if the action is relative to $\P$ and satisfies the following properties:
\begin{enumerate}
    \item The graph $K$ is fine and $\delta$--hyperbolic for some $\delta<\infty$.
    \item There are finitely many $\Gamma$--orbits
    of edges and each edge stabiliser is finite.
    \item The elements of $\P$ are representatives of the conjugacy classes of infinite vertex stabilisers of $K$.
\end{enumerate}
A group pair is \emph{relatively hyperbolic} if the pair admits a relatively hyperbolic action on a graph $K$.
The special case in which $\P$ is the empty collection reduces to the notion of a word hyperbolic group.
As our main results are already known in this case, we assume throughout this article that $\P$ is nonempty.  In this case, we may assume, in addition, that every vertex of $K$ has an infinite stabilizer (see \cite{Bowditch12_RelHyp}).
\end{definition}

\begin{definition}[Convergence groups]
For any compact Hausdorff space $M$, let $\Delta \subset M^n$ denote the subset of points with at least two coordinates equal.
The \emph{space of distinct $n$--tuples} $\Theta^n(M)$ is the quotient of $M^n \setminus \Delta$ by the symmetric group $S_n$, which acts by permuting the coordinates.
An action of a discrete group on $M$ is a \emph{convergence group action} if it satisfies one of the following conditions:
\begin{enumerate}
\item $M$ is empty, and $G$ is a finite group.
\item $M$ contains only one point, and $G$ is arbitrary.
\item $M$ contains exactly two points, and $G$ is virtually infinite cyclic.
\item $M$ contains at least three points and the action on $\Theta^3(M)$ is proper.
\end{enumerate}
A subgroup $H$ of a convergence group is \emph{elementary} if its limit set $\Lambda H$ contains at most two points.
If the limit set of $H$ has exactly one point, we say that $H$ is \emph{parabolic}, and if it has two points we say that $H$ is \emph{loxodromic}. Every loxodromic subgroup is virtually infinite cyclic.
We refer the reader to \cite{Tukia94,Tukia94Erratum,Bowditch99_Convergence} for more background on convergence groups in the setting of compact Hausdorff spaces that need not be metrizable.

A convergence group action of a group pair $(\Gamma,\mathbb{P})$ on a compact Hausdorff space $M$ is \emph{relatively hyperbolic} if the elements of $\mathbb{P}$ are representatives of the conjugacy classes of maximal parabolic subgroups and the action of $\Gamma$ on $\Theta^2(M)$ is cocompact.
(In the elementary case, $\Theta^2(M)$ has at most one point, so every action on it is considered to be cocompact.)
\end{definition}

\begin{definition}[Bowditch boundary]
\label{def:BowditchBoundary}
Suppose $(\Gamma,\mathbb{P})$ has a relatively hyperbolic action on the graph $K$.
Using a variation of a construction due to Bowditch, Gerasimov constructs a compact Hausdorff space $M = M(K)$ on which $(\Gamma,\mathbb{P})$ acts as a relatively hyperbolic convergence group
\cite{Bowditch12_RelHyp,Gerasimov12}.
Although a relatively hyperbolic pair $(\Gamma,\mathbb{P})$ admits relatively hyperbolic actions on many different graphs $K$, the associated compactum $M=M(K)$ is uniquely determined up to $\Gamma$--equivariant homeomorphism \cite{Bowditch12_RelHyp,GerasimovPotyagailo16}.
This unique compactum $M$ is the \emph{Bowditch boundary}, denoted $M=\boundary(\Gamma,\mathbb{P})$.
\end{definition}

\begin{remark}[On finite generation]
The definitions of relative hyperbolicity for a group pair $(\Gamma,\mathbb{P})$ above make no restriction on the cardinality of the discrete group $\Gamma$.
It is not clear whether there are interesting examples of uncountable relatively hyperbolic groups. 
In practice, one often restricts attention to the case when $\Gamma$ is finitely generated.
We note that if $(\Gamma,\mathbb{P})$ is relatively hyperbolic then $\Gamma$ is countable if and only if the Bowditch boundary $M= \boundary(\Gamma,\mathbb{P})$ is metrizable \cite{Gerasimov09}.
In fact, when $\Gamma$ is countable, $M$ is the (metrizable) Gromov boundary of a locally finite $\delta$--hyperbolic graph, known as the cusped space, whose geometry is central to many applications in the literature (see \cite{GrovesManning08,Hruska10}).
Cusped spaces do not play a significant role in this article, except for a brief appearance in the proof of Proposition~\ref{prop:PeanoNoCutPoint}.
\end{remark}

\subsection{Splittings of relatively hyperbolic groups}
\label{sec:Splittings}

This section summarizes the structure of splittings of relatively hyperbolic groups over parabolic subgroups.  The results of this section were first introduced in the finitely generated setting by Bowditch \cite{Bowditch01_Peripheral}.  This section extends that study to arbitrary relatively hyperbolic pairs and expresses the results in the terminology of Guirardel--Levitt's development of JSJ theory \cite{GuirardelLevitt_JSJ}.

Let $(\Gamma,\P)$ be a group pair, and let $\mathbb{A}$ be a family of subgroups of $\Gamma$ closed under conjugation.
Suppose $\Gamma$ acts on a simplicial tree $T$ without inversions and with no proper invariant subtree. Then $T$ is an \emph{$(\mathbb{A},\P)$--tree} if the action is relative to $\P$ and each edge stabilizer is a member of $\mathbb{A}$.
Trees $T$ and $T'$ are \emph{equivalent} if there are $\Gamma$--equivariant maps $T\to T'$ and also $T'\to T$.

\begin{definition}[Co-elementary]
Suppose $(\Gamma,\P)$ is relatively hyperbolic.
Two infinite elementary subgroups $H$ and $K$ are \emph{co-elementary} if they have the same limit set in the Bowditch boundary.
The limit point of a parabolic subgroup is a \emph{parabolic point}.  Its stabilizer is a maximal parabolic subgroup.
If $\{a,b\}$ is the limit set of a loxodromic subgroup $H$, then $H$ is a finite index subgroup of the maximal loxodromic subgroup $\Stab\{a,b\}$, which is again virtually cyclic.
In particular, co-elementary loxodromic subgroups are always commensurable.
It follows that each infinite elementary subgroup is contained in a unique maximal elementary subgroup.
For proofs of the above assertions, see \cite{Tukia94,Tukia94Erratum}.
\end{definition}

\begin{definition}
A group pair $(\Gamma,\P)$ is \emph{relatively one ended} if $\Gamma$ is infinite, $\Gamma \notin \P$, and $\Gamma$ does not split over a finite subgroup relative to $\P$.
\end{definition}

\begin{definition}[Cylinders]
Suppose $(\Gamma,\P)$ is relatively hyperbolic
and $\Gamma$ is one ended relative to $\P$. 
We will consider two cases in which $\mathbb{E}$ denotes either the collection of elementary subgroups or the collection of either finite or parabolic subgroups of $\Gamma$.
Let $\mathbb{E}_\infty$ denote the family of infinite subgroups that are members of $\mathbb{E}$.
The one-ended hypothesis implies that every $(\mathbb{E},\P)$--tree is actually an $(\mathbb{E}_\infty,\P)$--tree. 
Given an $(\mathbb{E},\mathbb{P})$--tree $T$, edges $e$ and $e'$ are considered to be equivalent if their stabilizers are co-elementary.
By \cite[Lem.~9.17]{GuirardelLevitt_JSJ}, the union of all edges in an equivalence class is a subtree of $T$.
(We note that the cited lemma is proved under the hypothesis that $\Gamma$ is finitely generated.  However, this hypothesis is not required, since all properties of elementary subgroups used in the proof also hold in the present more general setting.)
These subtrees are known as \emph{cylinders}.
\end{definition}

\begin{lemma}
\label{lem:Cylinders}
Let $T$ be an $(\mathbb{E},\P)$--tree.  Suppose $Y$ is the cylinder consisting of the equivalence class of the edge $e$.
Then the stabilizer $\Gamma_Y$ of $Y$ is the maximal elementary subgroup containing the stabilizer $\Gamma_e$.
Furthermore, $\Gamma_Y$ fixes a vertex of $T$.
\end{lemma}

\begin{proof}
Suppose $g$ stabilizes the cylinder $Y$. Then $\Gamma_e$ and $\Gamma_{g(e)}$ are co-elementary, so they have the same limit set $L=g(L)$. Thus, $g$ lies in the maximal elementary subgroup $\Stab(L)$.
Conversely, if $g\in \Stab(L)$ it must stabilize the cylinder $Y$, so that $\Gamma_Y=\Stab(L)$.

If $\Gamma_Y$ is a parabolic subgroup of $\Gamma$, then it has a fixed vertex in $T$ by the definition of $(\mathbb{E},\P)$--tree.
If $\Gamma_Y$ is loxodromic, then $\Gamma_e$ is a finite index subgroup of $\Gamma_Y$. Therefore, the edge $e$ has a finite orbit under the action of $\Gamma_Y$. Since $\Gamma_Y$ acts on $T$ with a bounded orbit, it must fix some vertex of $T$ (\cite[Prop.~I.4.19]{Serre_Trees}).
\end{proof}

\begin{definition}[Tree of cylinders]
The \emph{tree of cylinders} $T_c$ associated to $T$ is the bipartite tree with vertex set $V_0 (T_c) \sqcup V_1(T_c)$, where $V_0(T_c)$ is the set of all vertices $v$ of $T$ that belong to at least two cylinders, $V_1(T_c)$ is the set of cylinders $w$ of $T$, and there is an edge between $v$ and $w$ in $T_c$ if and only if $v$ is a vertex of the subtree $w$.
Equivalent trees have the same tree of cylinders \cite[Lem.~7.3]{GuirardelLevitt_JSJ}.
(Once again the cited result has a finite generation hypothesis that is not used in the proof.)

Furthermore, we will see that the graph $T_c$ is an $(\mathbb{E},\mathbb{P})$--tree equivalent to $T$.
To verify equivalence, one needs to show that each vertex stabilizer of $T$ fixes a vertex of $T_c$ and vice versa.  Let $v$ be a vertex of $T$ with stabilizer $\Gamma_v$. If $v$ lies in two cylinders, then $v$ is a $V_0$--vertex of $T_c$ fixed by $\Gamma_v$. On the other hand, if $v$ lies in a unique cylinder $Y$, then $Y$ is a $V_1$--vertex of $T_c$ fixed by $\Gamma_v$.
The converse is given by Lemma~\ref{lem:Cylinders}.
Therefore, $T$ and $T_c$ are equivalent.
It follows that each member of $\P$ has a fixed vertex in $T_c$. The stabilizer of an edge of $T_c$ is contained in the stabilizer of a cylinder, so it must be a member of $\mathbb{E}$ by Lemma~\ref{lem:Cylinders}.

It follows that $T_c$ is equal to its own tree of cylinders. In particular, $T$ and $T_c$ have the same family of nonelementary vertex stabilizers (\cite[Cor.~4.4]{GuirardelLevitt_Deformation}).
\end{definition}

\begin{proposition}[\cite{BigdelyWise13}, Rem.~2.2 and Lem.~4.9]
\label{prop:BigdelyWise}
Let $(\Gamma,\P)$ be relatively hyperbolic.  Suppose $T$ is a tree on which $\Gamma$ acts relative to $\P$.  If every edge stabilizer is a finitely generated elementary subgroup, then each vertex stabilizer is hyperbolic relative to a family of representatives of the conjugacy classes in $\Gamma_v$ of infinite subgroups of the form $\Gamma_v \cap gPg^{-1}$ for $g \in \Gamma$ and $P \in \P$.
\end{proposition}

We note that the statement of \cite[Lem.~4.9]{BigdelyWise13} requires that each vertex stabilizer and each edge stabilizer be finitely generated.  The finite generation of edge groups is used significantly in the proof, but the finite generation of vertex stabilizers is not used and may be omitted as a hypothesis.

\begin{proposition}[\emph{cf.}\ \cite{GuirardelLevitt_Splittings}, Lem.~3.8]
\label{prop:RelQCVertex}
Let $(\Gamma, \P)$ be relatively hyperbolic, and suppose $\Gamma$ is one ended relative to $\P$.
Let $\mathbb{E}$ be the family of elementary subgroups.
Let $T$ be any $(\mathbb{E},\P)$--tree that is equal to its own tree of cylinders.
Suppose each edge group of $T$ is finitely generated.
For each vertex $v\in V_0(T)$, the stabilizer $\Gamma_v$ is hyperbolic relative to 
a family of subgroups $\mathbb{Q}_v = \Inc_v \cup \mathbb{P}_{|\Gamma_v}$, defined as follows:
   \begin{enumerate}
       \item $\Inc_v$ consists of representatives of the conjugacy classes in $\Gamma_v$ of stabilizers of edges incident to $v$ in $T$.
       \item $\mathbb{P}_{|\Gamma_v}$ consists of representatives of the conjugacy classes in $\Gamma_v$ of subgroups $gPg^{-1}$ with $g \in \Gamma$ and $P \in \P$ that stabilize $v$ and do not stabilize any edge incident to $v$ in $T$.
   \end{enumerate}
\end{proposition}

\begin{proof}
By Proposition~\ref{prop:BigdelyWise}, each $\Gamma_v$ with $v \in V_0(T)$ is hyperbolic relative to a family of representatives of the conjugacy classes in $\Gamma_v$ of infinite subgroups of the form $\Gamma_v \cap gPg^{-1}$ for $g \in \Gamma$ and $P \in \P$.
Observe that the stabilizers of vertices adjacent to $v$ are distinct maximal elementary subgroups, so their pairwise intersections are finite.

Suppose $H = \Gamma_v \cap gPg^{-1}$ is infinite.  Then $gPg^{-1}$ stabilizes either $v$ or a vertex $w$ adjacent to $v$.  In either case, $H$ is a conjugate to a member of $\Inc_v \cup \mathbb{P}_{|\Gamma_v}$.
Conversely, if $H \in \Inc_v \cup \mathbb{P}_{|\Gamma_v}$ does not have the form $\Gamma_v \cap gPg^{-1}$, then $H$ must be a maximal elementary loxodromic subgroup of $\Gamma_v$.  A theorem of Osin \cite{Osin06_Elementary} states that such subgroups can be added to the peripheral structure of $\Gamma_v$.
\end{proof}

Given a parabolic splitting, one may construct a graph $K$ with a relatively hyperbolic action that is compatible with the splitting in the following sense.

\begin{definition}[Compatible action]
\label{defn:Compatible}
Suppose $(\Gamma,\P)$ is relatively hyperbolic and relatively one ended.
Let $\mathbb{E}$ be the family of finite or parabolic subgroups.
Let $T$ be an $(\mathbb{E},\P)$--tree with finitely generated edge stabilizers such that $T$ is equal to its tree of cylinders.
For each vertex $v \in V_0(T)$, choose a graph $K_v$ admitting a relatively hyperbolic action of the pair $(\Gamma_v,\Q_v)$. If $v$ and $v'$ are in the same $\Gamma$--orbit, we choose the same graph $K_v = K_{v'}$. For each $w \in V_1(T)$ let $K_w$ be a graph consisting of a single vertex.
For each edge $e$ of $T$ joining vertices $v$ and $w$, we glue the vertex $K_w$ to the unique vertex of $K_v$ fixed by the parabolic subgroup $\Gamma_e$.
The graph $K$ resulting from these identifications admits a natural relatively hyperbolic action of $(\Gamma,\mathbb{P})$ and
is said to be \emph{compatible} with $T$.
(See \cite[Thm.~1.4]{BigdelyWise13} for more details.)
\end{definition}

The following theorem due to Osin is often useful when studying splittings over parabolic subgroups.

\begin{theorem}[\cite{Osin06_RelHyp}, Thm.~2.44 and Prop.~2.29]
\label{thm:OsinFinGen}
Let $(\Gamma,\P)$ be relatively hyperbolic. If\/ $\Gamma$ does not split relative to $\P$ over parabolic subgroups, then $\Gamma$ and the members of $\P$ are all finitely generated.
\end{theorem}

The tree $T_{JSJ}$ in the following proposition is the \emph{canonical JSJ tree of cylinders} for splittings of $\Gamma$ over parabolic subgroups relative to $\P$.

\begin{proposition}
\label{prop:JSJfinitelygenerated}
Let $(\Gamma,\P)$ be relatively hyperbolic and relatively one ended.
Let $\mathbb{E}$ be the family of finite or parabolic subgroups.
There exists an $(\mathbb{E},\P)$--tree $T_{JSJ}$ equal to its own tree of cylinders having the following properties.

The action of $\Gamma$ on $T_{JSJ}$ has finitely many orbits of edges.
Each edge stabilizer is infinite, parabolic, and finitely generated. For each $v \in V_0(T_{JSJ})$ the stabilizer $\Gamma_v$ is finitely generated, the pair $(\Gamma_v,\Q_v)$ is relatively hyperbolic, and $\Gamma_v$ does not split over a finite or parabolic subgroup relative to $\Q_v$.
\end{proposition}

\begin{proof}
There exists a unique equivalence class, known as the JSJ deformation space, such that the tree of cylinders $T_{JSJ}$ of any tree in this class has the following properties.
The action on $T_{JSJ}$ has finitely many orbits of edges and, for each $v \in V_0(T_{JSJ})$, the stabilizer $\Gamma_v$ does not split relative to $\mathbb{Q}_v$ over any finite or parabolic subgroup.
(For the existence and uniqueness of the JSJ deformation space, see \cite[\S 3.8]{GuirardelLevitt_JSJ} under the assumption that $\Gamma$ is finitely generated. The proof never uses finite generation, so it holds in the present setting. For the fact that $\Gamma_v$ does not split, see \cite[Cor.~4.14]{GuirardelLevitt_JSJ}.)

Proposition~\ref{prop:RelQCVertex} implies that $\Gamma_v$ is hyperbolic relative to $\Q_v$ for each $v \in V_0(T_{JSJ})$.
Thus, by Theorem~\ref{thm:OsinFinGen} the group $\Gamma_v$ and all members of $\Q_v$ are finitely generated.
In particular, every edge incident to $v$ in $T_{JSJ}$ has a finitely generated stabilizer.
\end{proof}

\section{Constructing the Bowditch boundary}
\label{sec:ConstructingBoundary}

In this section, we study the definition of the Bowditch boundary, proving the equivalence of two constructions of the boundary due to Bowditch \cite{Bowditch12_RelHyp} and Gerasimov \cite{Gerasimov12}.
In each case, a fine $\delta$--hyperbolic graph $K$ leads to a canonical compact Hausdorff space that contains the vertex set $V(K)$ as a dense subspace.
Bowditch constructs a compact Hausdorff space $\Delta K$ whose underlying set is $V(K) \cup \boundary K$.
The topology on $\Delta K$ is given by an explicit neighborhood basis, which is quite simple and easy to work with directly.
However, the proof in \cite{Bowditch12_RelHyp} that $\Delta K$ is compact is valid only in the special case that $K$ has a countable vertex set.
Gerasimov gives a different definition of a compactum $M=M(K)$ for arbitrary $K$ by first defining a uniform structure on $V(K)$ and then letting $M$ be its Cauchy--Samuel completion, a space whose points are the minimal Cauchy filters on $V(K)$. The latter construction is theoretically elegant but quite technically elaborate, and the exact relationship between the compactum $M$ and the set $V(K) \cup \boundary K$ is a bit obscured by this approach.

In this section we show---with no restriction on the cardinality of $K$---that the two constructions lead to spaces that are canonically homeomorphic.
In order to produce this homeomorphism between $\Delta K$ and $M(K)$, the main step is Proposition~\ref{prop: Del K is compact}, which establishes the compactness of $\Delta K$ for an arbitrary fine $\delta$--hyperbolic graph $K$.
We begin with a study of the topology on $\Delta K$ introduced in \cite[\S 8]{Bowditch12_RelHyp}.

\begin{definition}[Topology on $\Delta K$]
\label{def: topology on Delta K}
Let $K$ be an arbitrary fine $\delta$--hyperbolic graph for some constant $\delta$. Let $\Delta K$ denote the disjoint union $V(K) \cup \boundary K$.
For points $a,b \in \Delta K$, we denote a geodesic from $a$ to $b$ by $[a, b]$. For each $a \in \Delta K$ and each finite subset $A \subseteq V(K)$, define
\[
   M'(a, A) = \bigset{b \in \Delta K}{\text{there exists $[a,b]$ such that $[a,b] \cap A \setminus {a} = \emptyset$} }.
\]
It is established in \cite[\S 8]{Bowditch12_RelHyp} that the sets $M'(a,A)$ form a base for a topology on $\Delta K$ such that for each $a$ the family $\bigl\{ M'(a,A)\}_A$ is a neighborhood base at $a$.
The space $\Delta K$ is Hausdorff and contains $V(K)$ as a dense subset.
\end{definition}

Suppose $f\colon \N \to \N$ is any function with $f\ge 1_{\N}$ and bounded above by a linear function. 
A \emph{simple $f$--quasigeodesic} is a simple path $\gamma$ such that for each subpath $\gamma'$ with endpoints $x$ and $y$ we have $\operatorname{length}(\gamma')\leq f\bigl( d(x, y) \bigr)$.
For each $f$ as above, equivalent neighborhood bases for a point $a \in \Delta K$ are given by the collections
\[
  \bigl\{ M'_f(a,A) \bigr\}_A \quad
  \text{and} \quad
  \bigl\{ P'_f(a,J) \bigr\}_J
\]
defined below, where $A$ ranges over all finite sets of vertices and $J$ ranges over all finite sets of edges of $K$. (When $a$ is a vertex, it suffices to let $J$ range over finite sets of edges incident to $a$.)
\begin{enumerate}
    \item The set $M'_f(a, A)$ contains all $b \in \Delta K$ such that vertex set of at least one simple $f$--quasigeodesic from $a$ to $b$ is disjoint from $A\setminus \{a\}$.
    \item The set $P'_f(a, J)$ contains all $b \in \Delta K$ such that edge set of at least one simple $f$--quasigeodesic from $a$ to $b$ does not intersect $J$.
\end{enumerate}

Shortcutting is a process that produces a simple quasigeodesic from a pair of geodesics as in the following lemma (see \cite[\S 8]{Bowditch12_RelHyp} for a proof).

\begin{lemma}[Shortcutting]
\label{lem: shortcutting}
There exists a function $f_1\colon \N \to \N$ with $f_1\ge 1$ and bounded above by a linear function and there exists a constant $R_1$, such that $f_1$ and $R_1$ depend only on the hyperbolicity constant of $K$ and such that the following holds.
Suppose $a$ and $b$ are two points in $\Delta K$  and $c \in V(K)$ such that both $a$ and $b$ are not the same point of $\partial K$. Let $\alpha$ be a geodesic from $a$ to $c$ and $\beta$ be a geodesic from $b$ to $c$.
Then there exists a simple $f_1$--quasigeodesic
$\gamma$ from $a$ to $b$ such that $\gamma = \alpha' \cup \omega \cup \beta'$, where the path $\alpha'$ is a subpath of $\alpha$, the path $\beta'$ is a subpath of $\beta$, and the path $\omega$ has length at most $R_1$. If $a\in V(K)$ then the edge of $\gamma$ incident to $a$ is contained in $\alpha$.
\end{lemma}

We write $M'_{f_1} (a, A)$ as $M'_1 (a, A)$ and so on.
Bowditch claims that the following result holds for arbitrary $K$ in \cite[Prop.~8.6]{Bowditch12_RelHyp}.  However, the proof depends, in an essential manner, on an unstated hypothesis that $K$ has a countable vertex set.  The following proposition extends Bowditch's result to the general case, removing this restriction on the cardinality of $K$.

\begin{proposition}
\label{prop: Del K is compact}
The space $\Delta K$ is compact.
\end{proposition}

\begin{proof}
We first note that, for each finite set $A$ of vertices of $K$, the finite collection $\mathcal{U}_A = \bigset{M'(a,A)}{a\in A}$ is an open cover of $\Delta(K)$.
Indeed, for each $x \in \Delta K$, there exists a geodesic from $x$ to some point $a\in A$ that contains no points of $A$ in its interior, so that $x \in M'(a,A)$.

Let $\mathcal{U}$ be an arbitrary open cover of $\Delta K$. 
In order to show that $\mathcal{U}$ has a finite subcover, it suffices to show that for some finite set $A\subset V(K)$ the finite cover $\mathcal{U}_A$ refines $\mathcal{U}$.
Suppose by way of contradiction that for every finite set $A\subset V(K)$, there is at least one \emph{rowdy} element $a_A \in A $ such that $M'(a_A, A)$ is not contained in any $U \in \mathcal{U}$.
Fix any vertex $x_0$ in $K$. For each finite $A$, choose a geodesic $\alpha_A$ from $x_0$ to $a_A$.

Observe that for each $a \in V(K)$, each $U \in \mathcal{U}$ containing $a$, and each finite set $B\subset V(K)$, there exists a finite set $A$ containing $B \cup \{a\}$ such that $M'(a,A) \subseteq U$.
Indeed, since $U$ is open, we may choose a finite set of vertices $F=F(a,U)$ such that $a \in M'(a, F) \subset U$. 
But then the finite set $A= F \cup B \cup \{a\}$ satisfies $M'(a, A) \subseteq M'(a,F) \subseteq U$. In particular, any finite set $B$ of vertices with rowdy element $a_B\in B$, can be increased to a larger finite set $A$ such that $a_B$ is no longer rowdy with respect to $A$.

We will construct a geodesic ray $\epsilon$ with vertices $(x_0,x_1,x_2,\dots)$ and a nested sequence of cofinal collections of finite vertex sets $\Lambda_0 \subseteq \Lambda_1 \subseteq \cdots$ such that, if $A\in \Lambda_n$, the geodesic $\alpha_A$ shares its first $n$ edges with $\epsilon$.
(A collection $\Lambda$ of finite sets is \emph{cofinal} if for each finite set $A$, some member of $\Lambda$ contains $A$.)
We construct $x_{n+1}$ and $\Lambda_{n+1}$ recursively, beginning with the previously chosen $x_0$ and the collection of all finite sets for $\Lambda_0$.

Assume that $x_{n}$ and $\Lambda_{n}$ have been previously constructed.
Clearly $x_n \in U$ for some $U \in \mathcal{U}$, so that $x_n$ has a basic neighborhood $P_1'(x_n, J_n) \subseteq U$ for some finite set of edges $J_n$ incident to $x_n$.
Let $\Lambda'_n \subseteq \Lambda_n$ be the subcollection of all finite sets of vertices $B$ large enough to contain all endpoints of the edges in $J_n$ and large enough that $x_n$ is not rowdy with respect to $B$.

Suppose $B\in \Lambda'_n$.  Then the geodesic $\alpha_B$ from $x_0$ to $a_B$ shares its first $n$ edges with $\epsilon$ and has length at least $n+1$.
We show that the edge of $\alpha_B$ following $x_n$ must be a member of the set $J_n$.
If not, then by Lemma~\ref{lem: shortcutting}, for each $b \in M'(a_B,B)$ there is a simple $f_1$--quasigeodesic $\gamma$ from $x_n$ to $b$ whose first edge is contained in $\alpha_B$ and, thus, is not contained in $J_n$.
In particular, the quasigeodesic $\gamma$ avoids $J_n$, so that $b \in P'_1(x_n, J_n)$. Consequently, $M'(a_B, B) \subseteq P'_1(x_n, J_n) \subseteq U$, contradicting the assumption that $a_B$ is rowdy in $B$. Therefore, the edge of $\alpha_B$ following $x_n$ must be a member of $J_n$ as claimed.
Since $J_n$ is finite, we may pass to a cofinal subcollection $\Lambda_{n+1} \subseteq \Lambda'_n$ such that all geodesics $\alpha_{C}$ with $C \in \Lambda_{n+1}$ involve the same edge of $J_n$. Letting $x_{n+1}$ be the other endpoint of this edge (distinct from $x_n$) completes the recursive construction of the ray $\epsilon$.

Let $y \in \partial K$ be the endpoint of $\epsilon$. There is some $U \in \mathcal{U}$ such that $ y \in U$. There exists some finite set $S \subset V(K)$ and a basic neighborhood $M'_1(y, S) \subseteq U$. As $S$ is finite, it lies in the ball of radius $r$ about $x_0$ for some $r < \infty$.

Choose $n > r + R_1 + \delta +  1$ where $R_1$ is the constant from Lemma~\ref{lem: shortcutting} and $\delta$ is the hyperbolicity constant of $K$.
Choose any $A \in \Lambda_n$ large enough that $S\subseteq A$.
We claim that $M'(a_A, A) \subseteq M'_1(y, S)$. 

Let $\alpha$ be a geodesic from $a_A$ to $y$. Consider the thin geodesic triangle $T_{thin}$ with vertices $y, x_n, a_A$ such that the geodesic $\alpha$ is the side of the triangle from $y$ to $a_A$, the side from $x_n$ to $a_A$ is contained in the geodesic from $x_0$ to $a_A$ that intersects $\epsilon$ up to length $n$ and the geodesic from $x_n$ to $y$ is contained in $\epsilon$. The geodesic $\alpha$ is in the $\delta$--neighborhood of the union of geodesics $[x_n,y]$ and $[x_n,a_A]$. In particular if $z$ is any vertex of $\alpha$ then $d(x_0, z) \geq n - \delta = r + R_1 + 1$. Hence $d(x_0, z) > r$, implying that $z \not \in S$.

Now consider any $x \in M'(a_A, A)$.
Since $S \subseteq A$, there exists a geodesic $\beta$ from $a_A$ to $x$ with vertex set disjoint from $S$. Applying Lemma~\ref{lem: shortcutting}, we obtain a simple $f_1$--quasigeodesic $\gamma = \alpha' \cup \omega \cup \beta'$ from $y$ to $x$ where $\alpha' \subseteq \alpha$ and $\beta' \subseteq \beta$ and $\omega$ has length at most $R_1$.
Then $\gamma$ misses $S$. Indeed, if $s \in S$ intersects $\gamma$ then $s \in V(\omega)$. Let $t_1$ be the vertex at which $\alpha'$ meets $\omega$.
Then $d(x_0, t_1) < d(x_0, s) + d(s, t_1) \leq r + R_1$, contradicting the fact that $d(x_0, t_1) > r + R_1 + 1$.  Hence $V(\gamma) \cap S = \emptyset$. Therefore $x \in M'_1(y, S)$, which establishes that $M'(a_A, A) \subseteq M'_1(y, S) \subseteq U$. 
However, this conclusion contradicts our initial assumption that $a_A$ is rowdy in $A$.
\end{proof}

Next we define a uniformity $\mathcal{D}$ on $\Delta K$ that induces the given topology.

\begin{definition}
\label{def: uniformity on Delta K}
Let $J$ be any finite set of edges in $K$. Define an entourage $D_J \subseteq \Delta K \times \Delta K$ that contains all pairs $(a,b)$ such that there exists a geodesic from $a$ to $b$ that does not contain an edge from $J$. In other words, $(a,b) \in D_J$ if and only if $b \in P'(a,J)$.
The proof of Proposition~\ref{prop: Del K is compact} shows that the 
family of entourages $D_J$ for all finite edge sets $J$ is a base for the canonical uniform structure $\mathcal{D}$ on the compact space $\Delta K$ (see \cite[\S 36]{Willard_Topology}).
\end{definition}

\begin{remark}[Gerasimov's definition]
To any fine $\delta$--hyperbolic graph $K$ we have associated a canonical compact Hausdorff space $\Delta K$ that contains the vertex set $V(K)$ as a dense subspace.
Using a different construction, Gerasimov associates a compact Hausdorff space $M=M(K)$ to $K$ such that $M$ also contains $V(K)$ as a dense subspace (see \cite{Gerasimov12} and \cite[\S 2]{GerasimovPotyagailo16}).
We note that the space $M$ of Gerasimov coincides with the space $\Delta K$ constructed above, in the sense that the identity map on $V(K)$ extends to a unique homeomorphism $M\to\Delta K$.

Indeed, the canonical uniform structure $\mathcal{D}$ on $\Delta K$ restricts to a uniform structure, also denoted $\mathcal{D}$, on $V(K)$.
By Proposition~\ref{prop: Del K is compact}, the space $\Delta K$ is compact and, thus, complete as a uniform space. 
As Gerasimov defines $M$ to be the Cauchy--Samuel completion of this same uniform space $\bigl( V(K),\mathcal{D} \bigr)$, the claim follows immediately from the uniqueness of completions.
\end{remark}

If $(\Gamma,\P)$ has a relatively hyperbolic action on the fine hyperbolic graph $K$, then the action of $\Gamma$ on its vertex set $V$ induces a relatively hyperbolic convergence group action of the pair $(\Gamma,\P)$ on $\Delta K = M(K)$ by \cite{Gerasimov12}.
Thus $\Delta K$ may be identified with the Bowditch boundary $\boundary(\Gamma,\P)$ using the uniqueness theorem of \cite{Bowditch12_RelHyp} in the finitely generated case and of \cite{GerasimovPotyagailo16} in general (see Definition~\ref{def:BowditchBoundary}).

\section{Trees of compacta}
\label{sec:TreesOfCompacta}

Suppose $(\Gamma,\P)$ is relatively hyperbolic and $\Gamma$ acts on a simplicial tree $T$ relative to $\P$ with infinite, finitely generated parabolic edge stabilizers.
In this section, we show that the boundary of $(\Gamma,\P)$ may be described as the limit of a tree of compacta formed by gluing the boundaries of the vertex groups along parabolic points in the pattern of the tree $T$ and then compactifying.
The main results of this section are inspired by the treelike structures associated to peripheral splittings in \cite{Bowditch01_Peripheral} and \cite{Dahmani03}.
The inverse limit construction is a minor variation of a construction due to \Swiatkowski\ \cite{Swiatkowski20}, which itself is modeled on the Ancel--Siebenmann characterization of trees of manifolds \cite{AncelSiebenmann85,Jakobsche91}.

\subsection{The limit of a tree system}

\begin{definition}
\label{def: tree system of metric compacta}
A \emph{tree system} $ \Theta$ of compacta consists of the following data:
\begin{enumerate}
    \item A bipartite tree $T$ with bipartition $V(T) = V_0(T) \sqcup V_1(T)$.
    \item To each vertex $v$ of $V_0(T)$, we associate a compact space $M_v$ called a \emph{component space}. For each $w \in V_1$ we associate a one-point space $M_w$, known as a \emph{peripheral point}.
    \item To each oriented edge $\vec{e} \in E(T)$, we associate a one-point space $\Sigma_{\vec{e}}$ and a map $i_{\vec{e}} \colon \Sigma_{\vec{e}} \to M_{t(\vec{e}\,)}$.
    We require that in each component space $M_v$ the adjacent edge spaces have distinct images.
\end{enumerate}
\end{definition}

Let $\bigcup \Theta$ denote the quotient space of $\bigsqcup_{v \in V(T)} M_v$ by the equivalence relation generated by $i_{\vec{e}}(\Sigma_{\vec{e}}) \sim i_{-\vec{e}} (\Sigma_{-\vec{e}})$ for all oriented edges $ \vec{e} \in E(T)$.

\begin{definition}
\label{def: subtree system}
Suppose $S$ is any subtree of $T$. Let $V(S) = V_0(S) \sqcup V_1(S)$ and $E(S)$ denote the set of vertices and the set of edges respectively of $S$.
Let $\Theta_S$ be the restriction of the tree system $\Theta$ to the subtree $S$.
\end{definition}

\begin{definition}[Limit of a tree system]
\label{def: limitOftreeSystem}
For each finite subtree $F$ of $T$, the \emph{partial union} is the compact space $M_F = \bigcup \Theta_F $.
For a pair of subtrees $F_1 \subseteq F_2$, the \emph{collapse map} $f_{F_1 F_2} \colon M_{F_2} \to M_{F_1}$ 
is the quotient map obtained by collapsing $M_v$ to a point for each $v \in V(F_2) - V(F_1)$.
Notice that $ f_{F_1 F_2}\of f_{F_2 F_3} = f_{F_1 F_3}$ whenever $ F_1 \subseteq F_2 \subseteq F_3 $. Hence the system of spaces $M_F$ and maps $f_{FF'}$ with $F \subseteq F' $ is an inverse system of compacta indexed by the poset of all finite subtrees $F$ of $T$. 

The \emph{limit} $\lim \Theta $ \emph{of the tree system} $ \Theta $ is the inverse limit of the above inverse system. Observe that $ \lim \Theta $ is compact and Hausdorff, since it is an inverse limit of compact Hausdorff spaces.
\end{definition}

\subsection{The boundary as a tree system}
\label{sec:BoundaryTreeSystem}

Suppose $(\Gamma,\P)$ is relatively hyperbolic.
Let $T$ be any simplicial tree on which $\Gamma$ acts relative to $\P$ with finitely many orbits of edges such that all edge stabilizers are infinite, parabolic, and finitely generated.
Suppose $T$ is equal to its tree of cylinders.
Fix a relatively hyperbolic action of the pair $(\Gamma,\P)$ on a fine $\delta$--hyperbolic graph $K$ that is compatible with this splitting in the sense of Definition~\ref{defn:Compatible}.

\begin{definition}
\label{def: ends}
A \emph{ray} in a tree is a subgraph isomorphic to $[0,\infty)$ with its natural structure as an infinite graph.
Two rays are \emph{equivalent} if their intersection is a ray. An equivalence class of rays is an \emph{end} of the tree.
\end{definition}

\begin{definition}
\label{def:TreeMap}
Let $K'$ be the barycentric subdivision of $K$.
We consider $K'$ and $T$ as length metric spaces in which each edge has length one.
Consider the map $f \colon K'\to T$ defined as follows, using the notation of Definition~\ref{defn:Compatible}.
Recall that $K$ is formed as a quotient of the disjoint union of graphs $K_v$ for each $v \in V_0(T)$ and one-point graphs $K_w$ for each $w \in V_1(T)$.
Each edge $e$ of $K$ lies in the image of $K_v$ for a unique vertex $v \in V_0(T)$.
We map the barycenter $x$ of the edge $e$ to this vertex $v$.
If $y$ is a vertex of $K$ contained in the image of a unique subgraph $K_v$, we also map the vertex $y$ to the vertex $v$.
If $y$ is contained in the image of $K_v$ for more than one vertex $v \in V_0(T)$, then $y$ is the image of the vertex $K_w$ for a unique vertex $w \in V_1(T)$, and $w$ is adjacent in $T$ to all such vertices $v$.
In this case, we map the vertex $y$ to the vertex $w$.
Observe that, whenever vertices $x,y$ of the barycentric subdivision $K'$ are adjacent, their images $f(x)$ and $f(y)$ either coincide or are adjacent. We extend $f$ linearly to the edges of $K'$ in the natural way.
\end{definition}

Let $\boundary_T K \subset \boundary K$ denote the set of points represented by geodesic rays that have bounded intersection with each subgraph $K_v$ of $K$.

\begin{lemma}
\label{lem:BoundaryOfTree}
The map $f \colon K' \to T$ induces a bijection $\boundary f \colon \boundary_T K \to \Ends(T)$.
\end{lemma}

\begin{proof}
The map $f$ is Lipschitz, by construction.
Since each subgraph $K_v$ is convex and its image in $T$ has diameter $2$, it follows from Kapovich--Rafi \cite[Prop.~2.5]{KapovichRafi14} that $f$ is alignment preserving in the sense of Dowdall--Taylor \cite{DowdallTaylor17}.
The result is a direct consequence of \cite[Thm.~3.2]{DowdallTaylor17}.
\end{proof}

The \emph{tree system associated to $T$} is the system $\Theta$ with underlying tree $T$ whose vertex spaces are the compacta $M_v = \Delta K_v$ for $v \in V(T)$ such that each edge map $i_{\vec{e}}$ has image equal to the unique point of $\Delta K_{t(\vec{e}\,)}$ fixed by $\Gamma_e$.
By the definition of the tree of cylinders, if $\vec{e}_0$ and $\vec{e}_1$ are distinct oriented edges of $T$ with terminal vertex $v$, the parabolic groups $\Gamma_{e_0}$ and $\Gamma_{e_1}$ are not co-elementary.
Therefore, the vertices of $K_v$ fixed by $\Gamma_{e_0}$ and $\Gamma_{e_1}$ are distinct.

\begin{proposition}
\label{prop:ClassificationOfPoints}
There is a function $\rho\colon \bigl( \bigcup \Theta \bigr)\cup \Ends(T) \to \Delta K$ with the following properties:
\begin{enumerate}
    \item $\rho$ is a bijection.
    \item $\rho$ is continuous on $\bigcup \Theta$.
    \item For each $v \in V(T)$ 
    the restriction of $\rho$ to $\Delta K_v$ is the natural inclusion $\Delta K_v \to \Delta K$.
    \item The restriction of $\rho$ to $\Ends(T)$ is the map $(\boundary f)^{-1} \colon \Ends(T) \to \boundary_T K$, where $\boundary f$ is the map given by Lemma~\ref{lem:BoundaryOfTree}.
\end{enumerate}
\end{proposition}

\begin{proof}
Since the subgraphs $K_v$ are convex, a geodesic ray of $K$ having unbounded intersection with some $K_v$ represents a point of $\boundary K_v$. Thus the complement in $\Delta K$ of the set $\boundary_T K$ equals the union of the sets $\Delta K_v$.

For each $v,v' \in V_0(T)$, the sets $\Delta K_v$ and $\Delta K_{v'}$ intersect precisely when the subgraphs $K_v$ and $K_{v'}$ intersect, in which case their intersection is a single vertex $w$ of $K$ and the corresponding vertex of $V_1(T)$ is incident to $v$ and $v'$. Thus the system of inclusions $\Delta K_v \to \Delta K$ induces a continuous injection $\bigcup \Theta \to \Delta K$ with image $\Delta K \setminus \boundary_T K$.
\end{proof}

\begin{definition}[Halfspaces]
Let $\vec{e}$ be any oriented edge of $T$.
Removing the open edge $e$ from $T$ leaves two components $T(\vec{e}\,)$ and $T(-\vec{e}\,)$, the first containing the vertex $t(\vec{e}\,)$ and the second containing $o(\vec{e}\,)$.
The \emph{halfspaces} of $\relbndry$ corresponding to $\vec{e}$ are the sets
$H(\vec{e}\,)$ and $H(-\vec{e}\,)$, where
\[
   H(\vec{e}\,) = \rho \Biggl(\Ends T(\vec{e}\,) \cup  {\bigcup_{v \in V(T(\vec{e}\,))} \!\!\!\!\Delta K_v} \Biggr).
\]
\end{definition}

\begin{proposition}
For each oriented edge $\vec{e}$ of $T$, the halfspaces $H(\vec{e}\,)$ and $H(-\vec{e}\,)$ are closed sets of $\Delta K$ with union $\Delta K$ and with intersection the one-point space $K_w$, where $w$ is the vertex of $V_1(T)$ incident to $e$.
\end{proposition}

\begin{proof}
According to Proposition~\ref{prop:ClassificationOfPoints}, it suffices to show the following.
Suppose $K$ is a fine $\delta$--hyperbolic graph that is the union of two fine $\delta$--hyperbolic graphs $K_1$ and $K_2$ whose intersection is a single vertex $x$.
Then $\Delta K_1$ and $\Delta K_2$ are closed sets with union $\Delta K$ and with intersection $\{x\}$.
The assertions about the union and intersection are evident since, as a set, $\Delta K$ is the union of the vertex set and the Gromov boundary of $K$.

For each $y \in \Delta K_2 \setminus \{x\}$, the basic open neighborhood $M'\bigl( y, \{x\} \bigr)$ is contained in $\Delta K_2 \setminus\{x\}$.
Thus $\Delta K_2 \setminus\{x\}$ is open in $\Delta K$, and its complement $\Delta K_1$ is closed.  By symmetry, $\Delta K_2$ is also closed.
\end{proof}

A family of subsets $\mathcal{C}$ of a topological space $X$ is \emph{null} if for each open cover $\mathcal{U}$ of $X$, all but finitely many members of $\mathcal{C}$ are $\mathcal{U}$--small.  A set is \emph{$\mathcal{U}$--small} if it is contained in some $U \in \mathcal{U}$.
The following is a variation of a well-known result in decomposition space theory.

\begin{proposition}
\label{prop:NullUSC}
If $\mathcal{C}$ is a null family of closed, pairwise disjoint subsets of a compact Hausdorff space $X$, the quotient space $X/\mathcal{C}$ is Hausdorff.
\end{proposition}

\begin{proof}
Let $\hat{\mathcal{C}}$ be the partition of $X$ consisting of the members of $\mathcal{C}$ and all singleton sets $\{x\}$ such that $x$ is not contained in any member of $\mathcal{C}$.
We first show that for each open set $U$ of $X$, the union $U^*$ of all members of $\hat{\mathcal{C}}$ contained in $U$ is an open set.
Suppose $U$ is open and $x \in U^*$.
Then $x \in C$ for some $C \in \hat{\mathcal{C}}$ with $C \subseteq U$.
By normality, $C$ has an open neighborhood $V$ with $\overline{V}\subseteq U$.
Consider the open cover $\mathcal{V} = \bigl\{X-\overline{V},V,U\setminus\{x\} \bigr\}$, and let $C_1,\dots,C_k$ be the finitely many members of $\mathcal{C} \setminus \{C\}$ that are not $\mathcal{V}$--small.
Then $V \setminus (C_1 \cup \cdots\cup C_k)$ is an open neighborhood of $x$ contained in $U^*$, so that $U^*$ is open.
Since $X$ is normal, it follows immediately that $X/\mathcal{C}=X/\hat{\mathcal{C}}$ is Hausdorff.
\end{proof}

\begin{proposition}
\label{prop:NullHalfspaces}
Let $\mathcal{U}$ be any open cover of $\Delta K$.
For all but finitely many oriented edges $\vec{e}$ of $T$, at least one of $H(\vec{e}\,)$ and $H(-\vec{e}\,)$ is $\mathcal{U}$--small.
\end{proposition}

\begin{proof}
Let $J$ be a finite set of edges of $K$, and let $D_J \subset \Delta K \times \Delta K$ be the associated entourage (see Definition~\ref{def: uniformity on Delta K}).
A set $Y \subset \Delta K$ has \emph{diameter less than $D_J$} if $(x,y) \in D_J$ for all $x,y \in Y$.
By compactness, it suffices to show the following holds for any $J$.
For all but finitely many edges $\vec{e}$ of $T$, one of the halfspaces $H(\vec{e}\,)$ and $H(-\vec{e}\,)$ has diameter less than $D_J$.

Let $S$ be the finite set of vertices $s\in V_0(T)$ such that $K_s$ contains an edge of $J$.
If an edge ${e}$ of $T$ is not contained in the convex hull $\bar{S}$ of $S$ then $e$ may be oriented so that the halftree $T(\vec{e}\,)$ is disjoint from $\bar{S}$.  For each such edge, the halfspace $H(\vec{e}\,)$ has diameter less than $D_J$, completing the proof since $\bar{S}$ has only finitely many edges.
\end{proof}

\begin{proposition}
\label{prop:TreeOfSpaces}
The Bowditch boundary $\Delta K$ is homeomorphic to the limit $\lim\Theta$ of the tree system associated to $T$.
\end{proposition}

\begin{proof}
For each finite subtree $F$ of $T$, we define a map $\mu_F \colon \Delta K \to M_F$, where $M_F$ is the corresponding partial union.
Let $E_F$ be the set of all oriented edges with $t(\vec{e}\,)$ in $T \setminus F$ and with $o(\vec{e}\,)$ in $F$.
We first show that the family of halfspaces $H(\vec{e}\,)$ with $\vec{e} \in E_F$ is null.
Let $\mathcal{U}$ be an arbitrary open cover of $\Delta K$.
Form a new cover $\mathcal{U}'$ as follows.
Let $a\ne b$ be two points of the partial union $M_F$.
Define $\mathcal{U}'$ to contain all open sets of $\mathcal{U}$, except that each set $U \in \mathcal{U}$ with $M_F \subset U$ is replaced by the pair of sets $U\setminus\{a\}$ and $U \setminus\{b\}$.
Then the cover $\mathcal{U}'$ refines $\mathcal{U}$, and $M_F$ is not $\mathcal{U}'$--small.
Nullity now follows by applying Proposition~\ref{prop:NullHalfspaces} to the family of edges $E_F$.

Let $\Delta K/{\sim}$ be the quotient in which each halfspace $H(\vec{e}\,)$ with $\vec{e} \in E_F$ is collapsed to a point.
The inclusion $M_F \to \Delta K$ induces a continuous bijection $M_F \to \Delta K/{\sim}$.  
According to Proposition~\ref{prop:NullUSC}, the quotient $\Delta K /{\sim}$ is Hausdorff, so this bijection is a homeomorphism.
We define $\mu_F$ to be the continuous retraction $\Delta K \to \Delta K/{\sim} \to M_F$.
Since the maps $\mu_F$ commute with the maps of the inverse system, they induce a continuous closed map $\mu\colon \Delta K \to \lim\Theta$.
As each $\mu_F$ is a surjective map from a compact space to a Hausdorff space, the induced map $\mu$ is surjective \cite[Thm.~3.2.14]{Engelking_Topology}.

To see that $\mu$ is injective, it suffices to show that
\[
   \mu \of \rho\colon {\textstyle \bigl(\bigcup \Theta\bigr)} \cup \Ends(T) \to \lim\Theta
\]
is injective, where $\rho$ is the bijection from Proposition~\ref{prop:ClassificationOfPoints}.
If $x\ne y$ lie in $\bigcup\Theta$ then there exists a finite subtree $F$ of $T$ such that $x,y \in M_F$.
Since $\mu_F$ is a retraction, we have $\mu_f\of\rho(x) =x \ne y = \mu_f\of\rho(y)$.
Thus $\mu\of\rho(x)\ne \mu\of\rho(y)$.

If $x\ne y$ lie in $\Ends(T)$, let $v \in V_0(T)$ be a vertex on the geodesic from $x$ to $y$.
Since the images of distinct edge spaces are distinct in $\Delta K_v$, 
it follows that $\mu_v\of\rho$ maps $x$ and $y$ to distinct points of $M_v=\Delta K_v$.
The case when $x \in \bigcup\Theta$ and $y\in \Ends(T)$ is similar.
In either case, we see that $\mu\of\rho(x)\ne \mu\of\rho(y)$.
Thus $\mu\of\rho$ is injective.
\end{proof}

\section{Connectedness properties of boundaries}
\label{sec:Connectedness}

This section discusses criteria for the connectedness and local connectedness of the Bowditch boundary.

Recall that a group pair $(\Gamma,\P)$ is relatively one ended if $\Gamma$ does not admit an action relative to $\P$ on a simplicial tree with finite edge stabilizers.
Relative one-endedness is weaker than absolute one-endedness.  For instance, let $\Gamma$ be the fundamental group of a genus one surface with one cusp and let $P$ be the fundamental group of the cusp.
Then $\Gamma$ is a free group of rank two, which is not one ended.  But $\Gamma$ is one ended relative to $P$, having a Bowditch boundary homeomorphic to the circle $S^1$.

\begin{proposition}
\label{prop: stallings}
Let $(\Gamma,\P)$ be relatively hyperbolic.
The boundary $M=\boundary(\Gamma,\P)$ is connected if and only if $\Gamma$ is one ended relative to $\P$.
\end{proposition}

This proposition was established in the finitely generated case by Bowditch \cite[Thm.~10.1]{Bowditch12_RelHyp}.
The proof of the forward implication is similar to that given by Bowditch.  Bowditch's argument for the reverse implication is valid only when $\Gamma$ is finitely generated.  The general case is easily reduced to the finitely generated case using Propositions \ref{prop:JSJfinitelygenerated} and \ref{prop:TreeOfSpaces}.

\begin{proof}
We first assume that $\Gamma$ admits a nontrivial action relative to $\P$ on a simplicial tree $T$ with finite edge stabilizers.  If $K$ is any fine hyperbolic graph with a relatively hyperbolic action  of $(\Gamma,\P)$, we may map the vertex set $V(K)$ equivariantly to the vertex set $V(T)$ of $T$.  Each edge $e$ of $T$, separates $T$ into two halftrees $T_1$ and $T_2$. Let $A_i$ be the set of vertices of $K$ mapping to the halftree $T_i$.
Then the set $J$ of edges of $K$ connecting a vertex of $A_1$ to a vertex of $A_2$ is a finite set.
If $a_1$ and $a_2$ are vertices of $A_1$ and $A_2$ respectively, then the basic open sets $P'(a_1,J)$ and $P'(a_2,J)$ are disjoint and partition $\Delta K = M$.
Thus $M$ is disconnected.

For the converse, assume that $\Gamma$ is one ended relative to $\P$.  Let $T_{JSJ}$ be the JSJ tree of cylinders over parabolic subgroups given by Proposition~\ref{prop:JSJfinitelygenerated}.
Then each vertex group $\Gamma_v$ with $v \in V_0(T_{JSJ})$ is finitely generated, each edge group is finitely generated, and the pair $(\Gamma_v,\Q_v)$ is relatively hyperbolic, where $\Q_v$ is the peripheral structure given by Proposition~\ref{prop:RelQCVertex}.
Furthermore, $\Gamma_v$ is one ended relative to $\Q_v$.

As $\Gamma_v$ is finitely generated, one-endedness of the relatively hyperbolic pair $(\Gamma_v,\Q_v)$ implies that its Bowditch boundary $\boundary(\Gamma_v,\Q_v)$ is connected by \cite[Prop.~10.1]{Bowditch12_RelHyp}.
According to Proposition~\ref{prop:TreeOfSpaces}, the boundary $\boundary(\Gamma,\P)$ is the inverse limit of a tree system whose vertex spaces are all connected.
In particular, the factor spaces $M_F$ of the inverse system are connected.
We conclude that the inverse limit is connected, since an inverse limit of connected compact Hausdorff spaces is always connected \cite[Thm.~6.1.20]{Engelking_Topology}. 
\end{proof}

The following elementary result of Capel provides a criterion for the local connectedness of an inverse limit of compacta.

\begin{theorem}[\cite{Capel54}]
\label{thm: monotoneBondingCapel}
Let $\{X_\alpha\}$ be an inverse system such that each bonding map $X_\alpha\to X_{\alpha'}$ is monotone. If each factor space $X_\alpha$ is compact and locally connected then the inverse limit $\varprojlim X_\alpha$ is locally connected.
\end{theorem}

Recall that a map $f \colon Y \to Z $ is \emph{monotone} if $f$ is surjective and for each $z \in Z$ the set $f^{-1}(z)$ is compact and connected.

\begin{theorem}
\label{thm: localconnectednessInfiniteCase}
Let $(\Gamma, \P)$
be a relatively hyperbolic group pair, and suppose the boundary $M=\boundary(\Gamma,\P)$ is connected.
Let $T_c$ be the tree of cylinders for a splitting of $\Gamma$ relative to $\P$ over parabolic subgroups.
Suppose for each $v \in V_0(T_c)$ that the Bowditch boundary $\boundary(\Gamma_v,\Q_v)$ is locally connected.
Then $M$ is also locally connected.
\end{theorem}

\begin{proof}
By Proposition~\ref{prop:TreeOfSpaces}, the boundary $M$ is an inverse limit of a tree system $\Theta$ corresponding to the splitting $T_c$. 
For each finite subtree $F$ of $T_c$, the partial union $K_F$ is formed by gluing finitely many component spaces $K_v$ in the pattern of $F$.
Since each $K_v$ is connected, $K_F$ is again connected.
As a quotient of a compact, locally connected space, $K_F$ is also compact and locally connected.

To apply Theorem~\ref{thm: monotoneBondingCapel}, we only need to show that the bonding map $K_{F_2} \to K_{F_1}$ is monotone whenever $F_1$ is a subtree of $F_2$.
But, according to Definition~\ref{def: limitOftreeSystem}, each nontrivial point preimage has the form $K_F$ for some finite subtree $F$, and all such sets are compact and connected.
\end{proof}

\section{Pretrees and non-nesting actions}
\label{sec:Pretrees}

This section reviews the definition and basic properties of pretrees, a general treelike structure explored extensively by Bowditch in \cite{Bowditch99_Treelike}.
A key result of this section is Lemma~\ref{lem: non-nesting in quotient}, which states that a non-nesting action on a pretree induces a non-nesting action on its quotient by any equivariant full relation.
The results of this section lay the foundations for applications to the cut point structure of boundaries in Sections \ref{sec:CutPointPretree} and~\ref{sec:  minimalcodense}.

\begin{defn}
Let $K$ be a set with a ternary relation called \emph{betweenness}.
If $y$ is between $x$ and $z$, we write $xyz$.
Such a structure is a \emph{pretree} if it satisfies the following:
\begin{enumerate}
\item If $y$ is between $x$ and $z$, then $x\ne z$.
\item $y$ is between $x$ and $z$ if and only if $y$ is between $z$ and $x$.
\item If $y$ is between $x$ and $z$, then $z$ is not between $x$ and $y$.
\item If $xyz$ and $y \neq w $  then either $xyw$ or $wyz$. 
\end{enumerate}
\end{defn}

For points $x,y$ in a pretree $K$, the \emph{open interval} $(x,y)$ is the set of all $z$ between $x$ and $y$. The \emph{closed interval} $[x,y]$ is the set $(x,y) \cup \{x,y\}$.  Half-open intervals are defined similarly.
Distinct points $x\ne y$ are \emph{adjacent} if the interval $(x,y)$ is empty.
A pretree is \emph{discrete} if every interval $(x,y)$ is finite.
A subset $Q$ of a pretree $K$ is \emph{dense} in $K$ if for all distinct $x, y \in K$ the interval $(x,y)$ contains at least one point of $Q$.
We say $K$ is \emph{dense} if it is dense in itself.

A subset $A$ of a pretree $K$ is \emph{full} if $ [x, y] \subseteq A $ for all $x, y \in A$. An $arc$ of a pretree is a nonempty full subset $A$ such that for all distinct $x, y, z \in A$, we have $xyz$ or $yzx$ or $zxy$.
Note that every interval is an arc.
A \emph{direction} on an arc $A$ is a linear order $<$ on $A$ such that $xyz$ holds if and only if either $x<y<z$ or $z<y<x$. We refer to $(A,<)$ as a \emph{directed} arc.
Any arc with at least two elements admits precisely two directions.
A \emph{final segment} of a directed arc $A$ is a subset $B \subseteq A$ such that if $x \in B$ and $y \in A$ with $x\le y$, then $y \in B$. Two directed arcs $A$ and $A'$ are \emph{cofinal} if they have a common final segment.

A pretree is \emph{complete} if every arc is an interval. A pretree is \emph{interval complete} if every arc contained in an interval is itself an interval.
We note that if $K$ is a complete pretree, then any full subpretree of $K$ is interval complete.

Suppose $K$ is a pretree.
The \emph{median} $c = \text{med}(x,y,z)$ of a triple of points $x,y,z \in K$ is the unique point $c \in [x, y] \cap [y, z] \cap [z, x]$.
A \emph{median pretree} is a pretree in which every triple has a median.

\begin{definition}[Connector points]
Let $K$ be any complete median pretree.
Let $X$ and $Y$ be nonempty, full, disjoint subsets of $K$ such that $X\cup Y$ is also full.
A point $p(Y, X) \in Y$ is a \emph{connector point} of the ordered pair $(Y, X)$ in $K$ if we have $[x, p] \cap Y = \{p\}$ for all $x \in X$.
\end{definition}

Existence and uniqueness of connector points is given by the following lemma due to Bowditch.

\begin{lemma}[\cite{Bowditch99_Treelike}, Lem.~6.13]
\label{lem:Connector}
Let $K$ be a complete median pretree, and let $X$ and $Y$ be disjoint, full, nonempty subsets such that $X \cup Y$ is full.
If the connector point $p(Y,X)$ exists, it is unique.
Furthermore, at least one of the pairs $(X,Y)$ and $(Y,X)$ has a connector point.
\end{lemma}

A \emph{full relation} on a pretree $K$ is an equivalence relation for which every equivalence class is full.
If $\sim $ is a full equivalence relation, $[x]_\sim$ denotes the equivalence class of the point $x$.
The set $K/{\sim}$ has a natural induced pretree structure such that $Y$ is between $X$ and $Z$ if and only if there exists $y_0 \in Y$ such that $x y_0 z$ for all $x\in X$ and $z \in Z$. If a pretree $K$ is median or complete, then so is $K/{\sim}$ (see \cite[Lem.~4.2]{Bowditch99_Treelike} for details).
A point $x$ in a pretree $K$ is \emph{terminal} if there
exists no $y, z \in K$ such that $yxz$. The \textit{interior} $K^o$ of a pretree $K$ is the subpretree obtained by removing all terminal points. 

\begin{lemma}
\label{lem: interior parabolic point of pretree}
Let ${K}$ be any pretree with a full relation $\sim$.
There exists a full subpretree $K'$ contained in the interior $K^o$ that is invariant under all automorphisms of $K$ and maps onto the interior of $K/{\sim}$.
\end{lemma}

\begin{proof}
If $X \in {K}/{\sim}$ lies between points $Y$ and $Z$, then by definition, there exists $x \in X$ lying between every pair $y, z \in {K}$ with $y \in Y$ and $z \in Z$. Thus, $x$ is an interior point of ${K}$ mapping to $\bar{x}$.
Let $K'$ be the smallest full subset of $K$ containing all such points $x$ and invariant under all automorphisms of $K$. Then $K'$ has the desired properties.
\end{proof}

An action of a group $\Gamma$ on a pretree $K$ is \emph{non-nesting} if no $g\in \Gamma$ maps an interval of $K$ properly inside itself. 

\begin{lemma}
\label{lem: non-nesting in quotient}
If $\Gamma$ has a non-nesting action on a complete median pretree ${K}$, then the induced action on ${K} / {\sim}$ is non-nesting for any equivariant full relation $\sim$.
\end{lemma}

\begin{proof}
By way of contradiction, suppose $g\bigl( [X, Y] \bigr) \subsetneq [X, Y]$ for some $g \in \Gamma$ and some interval $[X, Y] \subseteq K/{\sim}$. 
Replacing $g$ with $g^2$, we may assume that the restriction $g\big|[X,Y]$ is orientation preserving.  We consider two cases, depending on whether $g$ fixes an endpoint of $[X,Y]$.

Suppose $g(X) \in \bigl( X,g(Y) \bigr)$ and $g(Y) \in \bigl( g(X),Y \bigr)$.
Then there exist $a \in g(X)$ and $b \in g(Y)$ with $a \in \bigl( g^{-1}(a),b \bigr)$ and $b \in \bigl( a,g^{-1}(b) \bigr)$.
But then $(a,b) \subsetneq \bigl( g^{-1}(a), g^{-1}(b) \bigr)$, contradicting that $\Gamma$ is non-nesting on $K$.

On the other hand, suppose $g(X)=X$.  Then $g(Y)\in ( X,Y)$, so we may choose $b \in Y$ so that $g(b)$ is between $x$ and $y$ for any $x \in X$ and $y\in Y$.
Let
\[
   T = \bigset{t\in K\setminus (X \sqcup Y)}{\text{$t \in (x,y)$, for some  $x \in X$ and $y \in Y$}}.
\]
We show that $T$ is full. Suppose $t_1, t_2 \in T $ and choose $r \in (t_1,t_2)$. Since $t_i \in T $ there exist points $x_i\in X$ and $y_i \in Y$ such that $x_i t_i y_i$ holds. Suppose $m_1 = \text{med}(x_1, x_2, t_1)$ and $m_2 = \text{med}(y_1, y_2, t_2)$. Since $X$ and $Y$ are full, we have $m_1 \in X$ and $m_2 \in Y$. 
Since we have $x_2 t_2 y_2$, $t_2 m_2 y_2$, and $t_2 m_2 y_1$, we also have $x_2 t_2 m_2$ and, thus, $t_2 \in (x_2, y_1)$.
A similar argument establishes that $t_1 \in (x_2,y_1)$, so that $r\in(x_2,y_1)$.
Since $t_1$ and $t_2$ are not members of $X$, it follows that $r \notin X$. Similarly $r \not \in Y$. Hence $r \in T$, so that $T$ is full.

Notice that $T$ is non-empty as it contains $g(b)$. To see that $X\cup T$ is full, choose $x \in X$ and $t \in T$ and $s \in (x,t)$. Since $t \in (x,y)$ for some $y \in Y$, we have $s \in (x,y)$ by \cite[Lem.~2.3]{Bowditch99_Treelike}.
Clearly if $s$ is not a member of $X$ then $s$ is a member of $T$ by definition of $T$.
Since the members of $T$ belong to different equivalence classes (possibly multiple equivalence classes) than the members of $X$, the connector point $p$ of $X, T$ is unique.
Since $X$ is stabilized as a set by $g$, the connector point $p$ is fixed by $g$. This leads to a contradiction, since $\bigl[ g(p),g(b) \bigr]$ is a proper subset of $[p, b]$ in $K$. 
\end{proof}

Let $\Gamma$ be a group acting on a pretree $K$. The action is \emph{trivial} if some $x \in K$ is fixed by the whole group. 
An automorphism $g$ of $K$ is \emph{elliptic} if the fixed point set $Fix (g)$ is nonempty, it is an \emph{inversion} if some pair of adjacent points is transposed by $g$.

\begin{lemma}
\label{lem:ParabolicInterior}
Suppose $P$ is a finitely generated group with a non-nesting action on an interval complete, median pretree $K$. If $P$ fixes a point of $K$ and $K'$ is any $P$--invariant, full subpretree of $K$, then $P$ fixes a point of $K'$.
\end{lemma}

In \cite[\S I.6]{Serre_Trees}, Serre establishes several fundamental results about automorphisms of simplicial trees with fixed points. Bowditch--Crisp have extended several of these results to the pretree setting in \cite{BowditchCrisp01}.  Consequently the following proof based on Serre's results is valid in the pretree setting.

\begin{proof}
Choose $p \in P$. Since the action is non-nesting, $\Fix(p)$ is full and for any $x \in K$ the interval $[x,px]$ intersects $\Fix(p)$ in precisely one point by \cite[Lemma 3.7]{BowditchCrisp01}.
If we choose $x \in K'$, then the fixed point $y$ is also in $K'$, since $K'$ is full.
Hence, each element of $P$ has a fixed point in $K'$.
As $K'$ is also interval complete and median and the action of $P$ on $K'$ is non-nesting, the proof of \cite[Lemma 5.2]{BowditchCrisp01} implies that $P$ has a fixed point in $K'$.
\end{proof}

\section{The cut point pretree}
\label{sec:CutPointPretree}

In this section, we review a construction due to Swenson \cite{Swenson00} of a cut point pretree and cut point tree dual to a family of cut points in a continuum.
We establish Proposition~\ref{prop:CutPointPretree}, which describes the structure of this pretree when the continuum admits a convergence group action.

Recall that if $M$ is a compact Hausdorff space with at least three points, a \emph{convergence group action} on $M$ is an action such that the induced action on the space of distinct triples $\Theta^3(M)$ is proper.
The action is \emph{minimal} if $M$ has no proper subset that is nonempty, closed, and $\Gamma$--invariant.

The main application of the results below is to the Bowditch boundary of a finitely generated relatively hyperbolic group.  The compact Hausdorff topology on the boundary $\boundary(\Gamma,\P)=\Delta K$ defined in Section~\ref{sec:ConstructingBoundary} has a countable base whenever $\Gamma$ is countable, so the boundary is metrizable in this setting.

In the sequel, we study convergence actions only on metrizable spaces, usually requiring connectedness as well. We refer to Tukia \cite{Tukia94} for the basic structure theory of convergence groups on metrizable spaces.

\begin{definition}[Continua]
\label{def:Continuum}
A \emph{continuum} is a nonempty, compact, connected, metrizable topological space.
A continuum is \emph{nontrivial} if it contains more than one point, in which case it has the cardinality $\mathfrak{c}$ of the real line. Indeed, any continuum embeds in the Hilbert cube by Urysohn's theorem, and any nonempty open set in a nontrivial continuum contains a Cantor set (see, for instance, \cite[Thm.~6.2]{Kechris95} or \cite[Prob.~4.5.5]{Engelking_Topology}).
Note that any continuum containing a cut point must be nontrivial.
\end{definition}

\begin{definition}[$\R$--trees]
An \emph{$\R$--tree} is a uniquely arc-connected geodesic metric space, or equivalently a geodesic space with no subspace homeomorphic to the circle.
A \emph{topological $\R$--tree} is the underlying topological space of an $\R$--tree.
A \emph{dendrite} is a compact topological $\R$--tree.
\end{definition}

For each continuum $M$ equipped with a preferred family of cut points $C$, a pretree ${K}$ dual to $C$ is constructed in \cite{Swenson00} and described briefly below.

\begin{definition}[Cut point pretree]
\label{def:CutpointPretree}
Let $M$ be a continuum, and let $C$ be a nonempty family of cut points of $M$.
A pretree structure on $M$ is given by the following betweenness relation:
$c$ is \emph{between} $a$ and $b$ if $c \in C$ and $M\setminus \{c\}$ has a separation into disjoint nonempty open sets $U$ and $V$ with $a\in U$ and $b\in V$.
Consider the following  equivalence relation on $M$.
\begin{enumerate}
    \item If $c \in C$, then $c$ is equivalent to only itself.
    \item If $a, b \in M \setminus C$ then $a \sim b$ if $(a,b)$ is empty.
\end{enumerate}
Let ${K}$ denote the set of equivalence classes of this relation.
The set ${K}$ has an induced pretree structure given by the following betweenness relation:
\begin{enumerate}
    \item If $x,y \in K$ and $c\in C$, we say that $\{c\}$ is between $x$ and $y$ if for any $a \in x$ and $b \in y$ we have $c \in (a,b)$.
    \item If $x,y,z \in {K}$ with $z$ not a cut point, we say that $z$ is between $x$ and $y$ if for all $a \in x$, \ $b \in y$, and $d \in z$ we have $[a,d)\cap (d,b]=\emptyset$.
\end{enumerate}
The pretree ${K}$ is a complete median pretree (see \cite{Swenson00}), called the \emph{cut point pretree of $M$ dual to $C$}.
\end{definition}

\begin{lemma}[\cite{Swenson00}, Lem.~4]
\label{lem:KAdjacent}
\leavevmode
\begin{enumerate}
    \item
    \label{item:KAdjacent:Cyclic}
    If $M'$ is a nontrivial subcontinuum of $M$ not separated by any point of $C$, then the set $M'\setminus C$ is nonempty, dense in $M'$, and contained in a single equivalence class of $K$.
    \item
    \label{item:KAdjacent:Closure}
    If $x$ and $y$ are adjacent in $K$ then exactly one of them is a singleton set $\{c\}$ with $c \in C$ and the other of them is a non-singleton equivalence class whose closure contains the point $c$.
\end{enumerate}
\end{lemma}

The proof of the lemma is implicit in the proof of \cite[Lem.~4]{Swenson00}.  We give the details for the benefit of the reader.

\begin{proof}
Suppose the subcontinuum $M'$ is nontrivial and not separated by any point of $C$.
We first show that $M' \cap C$ is countable.
For each $c \in M' \cap C$, choose subcontinua $A_c$ and $B_c$ of $M$ such that $A_c \cup B_c = M$, \ $A_c \cap B_c = \{c\}$, and $M' \subseteq A_c$.  Then $\bigset{ B_c \setminus \{c\} }{c \in M' \cap C}$ is a collection of pairwise disjoint nonempty open sets of $M$.  Since $M$ is second countable, this family must be countable.
Any nonempty open set of $M'$ is uncountable (see Definition~\ref{def:Continuum}), so $M'\setminus C$ intersects every nonempty open set of $M'$; \emph{i.e.}, $M'-C$ is dense in $M'$.  By hypothesis, all points of $M'\setminus C$ are equivalent, establishing (\ref{item:KAdjacent:Cyclic}).

To see (\ref{item:KAdjacent:Closure}), suppose $x$ and $y$ are distinct members of $K$.  If neither $x$ nor $y$ is in $C$, then an element of $C$ must separate them, as $x\ne y$, so they are not adjacent.  On the other hand, if $x,y \in C$, then $M$ contains a minimal subcontinuum $M'$ containing $\{x,y\}$ by Zorn's Lemma.  If $x$ and $y$ are not separated by any point of $C$, then by minimality $M'$ is also not separated by any point of $C$. Therefore, $M'\setminus C$ is a nonempty set by (\ref{item:KAdjacent:Cyclic}). If $z$ is any $K$--equivalence class intersecting $M'\setminus C$, then $z$ is between $x$ and $y$.

If $x$ and $y$ are adjacent in $K$, the only possibility is that exactly one of them, say $x$, is a member of $C$.
It remains to show that $x$ is in the closure of the set $y$. If $M'$ is any minimal subcontinuum of $M$ containing $x \cup y$, then $M'$ is not separated by any point of $C$. Thus, $y = M'\setminus C$ and $x$ is in the closure $M'$ of $y$ by (\ref{item:KAdjacent:Cyclic}).
\end{proof}

\begin{corollary}
\label{cor:KWithoutInversion}
Any homeomorphism of $M$ leaving $C$ invariant induces a pretree automorphism of $K$ that is not an inversion.
\end{corollary}

We will need to use the following result due to Swenson.

\begin{lemma}[\cite{Swenson00}, Lem.~7]
\label{lem:SwensonLoxodromic}
Let $\Gamma$ have a convergence group action on $M$, and let $g\in \Gamma$ be loxodromic with fixed points $a$ and $b$ in $M$.
Let $K$ be the pretree of $M$ dual to a $\Gamma$--invariant family of cut points $C$.
\begin{enumerate}
    \item
    \label{item:LoxodromicTerminal}
    If $a$ and $b$ are separated by an element of $C$, then the equivalence classes $[a]$ and $[b]$ in $K$ are terminal, and $[a]$ and $[b]$ are the only fixed points of $g$ in $K$.
    \item
    \label{item:LoxodromicNonterminal}
    Suppose $a$ and $b$ are not separated by an element of $C$.
    If the open interval $\bigl( [a],[b] \bigr)$ is nonempty, then it contains only one element $[x]$ and we have $a,b\in C$ and $x \notin C$.
\end{enumerate}
\end{lemma}

The next lemma gives a stronger conclusion in the second case of the preceding lemma.

\begin{lemma}
\label{lem:SwensonLoxodromicfollowup}
Let $\Gamma$ have a convergence group action on $M$, and let $g\in \Gamma$ be loxodromic with fixed points $a$ and $b$ in $M$.
Let $K$ be the pretree of $M$ dual to a $\Gamma$--invariant family of cut points $C$.  Suppose $a$ and $b$ are not separated by an element of $C$. Then the set of fixed points of $g$ in $K$ is equal to the closed interval $\bigl[ [a],[b] \bigr]$.
\end{lemma}

\begin{proof}
Since there is at most one member of $K$ between $[a]$ and $[b]$, such an element must be fixed by $g$.  Conversely, suppose $Y \not\in \bigl[ [a], [b] \bigr]$ is fixed by $g$. The set $Y \subseteq M$ cannot be a single point, since $a$ and $b$ are the only fixed points of $g$ in $M$. Therefore, $Y$ must be a nontrivial subset of $M$ stabilized by $g$ and not separated by any member of $C$. Let $J$ be any compact subset of $Y$. Since $g$ is a loxodromic, the sets $g^n(J)$ and $g^{-n}(J)$ accumulate at the points $a$ and $b$ as $n \to \infty$. Thus, $a$ and $b$ each lie in the closure of $Y$.

We claim that $a$ must be a cut point of $C$. If not, then there would be a cut point $c \in C$ that separates the nontrivial classes $[a]$ and $Y$.
Then $M$ is the union of two subcontinua $A_c$ and $B_c$ with $a\in A_c$ and $Y\subseteq B_c$ and $A_c \cap B_c = \{c\}$. The closed set $B_c$ contains $Y$ but does not contain $a$, which is a contradiction.  Thus, $a$ is a cut point.  Similarly $b$ is also a cut point.

We claim that $[a]$ and $Y$ are adjacent in $K$.
The argument of the previous paragraph shows that there can be no cut point between $[a]$ and $Y$.
Suppose $Z\in K$ is not equal to a cut point of $C$ and $Z$ lies between $[a]$ and $Y$.
Then for any $b\in Y$ and $d\in Z$ there must be a cut point $c \in C$ between the points $b$ and $d$.
Express $M$ as the union of subcontinua $B_c$ and $D_c$ such that $b\in B_c$ and $d\in D_c$ and $B_c \cap D_c = \{c\}$.
We must have $Y\subseteq B_c$ and $Z \subseteq D_c$, since $Y$ and $Z$ cannot be separated by any member of $C$.
Since $a \in \overline{Y}$, we must have $a \in B_c$.
Thus $c$ separates $a$ from $d$, so that the intersection $[a,d) \cap (d,b]$ is not empty.
We have contradicted our assumption that $Z$ lies between $[a]$ and $Y$, so $[a]$ and $Y$ must be adjacent.
A similar argument shows that $[b]$ and $Y$ are adjacent, so $Y$ lies in the interval $\bigl( [a],[b] \bigr)$.
\end{proof}

When $C$ is countable, the following construction due to Swenson shows that the cut point pretree may be extended to form a tree by filling the gaps between adjacent elements. (See \cite{Swenson00,BowditchCrisp01,PapasogluSwenson06} for more details.)

\begin{definition}[Cut point tree]
Define
\[
   \widehat{K} = {K} \cup \bigcup_{\textrm{$\{x,y\}$ adjacent}} I(x,y)
\]
where $I(x,y)$ is a copy of the real open interval $(0,1)$.  We take $I(y,x)$ to be the same arc as $I(x,y)$, but with opposite orientation.
There exists a canonical extension of the interval relation to $\widehat{K}$ that endows $\widehat{K}$ with the structure of a dense, complete pretree such that each $I(x,y) = (x,y)$.
Each homeomorphism of $M$ leaving $C$ invariant induces an automorphism of $\widehat{K}$ by extending linearly to the intervals $I(x,y)$.
Furthermore, every interval of $\widehat{K}$ is order-isomorphic to an interval of the real line \cite[Thm.~13]{PapasogluSwenson06}.
The pretree $\widehat{K}$ is called the \emph{cut point tree} of $M$ dual to $C$ because it admits a canonical topology as an $\R$--tree such that every homeomorphism of $M$ induces a homeomorphism of $\widehat{K}$ (see \cite[Thm.~14]{PapasogluSwenson06}).
\end{definition}

If $\Gamma$ has a convergence group action on $M$, and $C$ is an equivariant family of cut points, Swenson constructs an $\R$--tree on which $\Gamma$ acts with several useful dynamical properties \cite[Thm.~8]{Swenson00}.

\begin{theorem}[\cite{Swenson00}, Thm.~8]
If $\Gamma$ is a countable convergence group without any infinite torsion subgroup acting on a continuum $M = \Lambda \Gamma$ with cut points, then $\Gamma$ has a non-nesting stable action on a topological $\R$--tree $R$. Furthermore the stabilizer of an arc will be an elementary subgroup of $\Gamma$ (viewed as a convergence group), and no point of $R$ is fixed by $\Gamma$.
\end{theorem}

We extend this theorem to give the following stronger conclusion involving the induced action on the cut point pretree, a conclusion that is implicit in Swenson's methods of proof.

\begin{proposition}
\label{prop:CutPointPretree}
Let $(\Gamma, \mathbb{P})$ be a group pair such that $\Gamma$ is countable.
Suppose $\Gamma$ has a minimal convergence group action relative to $\P$ on a continuum $M$. Let $C$ be a $\Gamma$--invariant family of cut points in $M$. Suppose ${K}$ is the cut point pretree of $M$ dual to $C$.
Then the action of\/ $\Gamma$ on $K$ is nontrivial and without inversions.
The action on the interior ${K}^o$ of ${K}$ is a non-nesting action such that the stabilizer of each interval of ${K}^o$ is elementary in the convergence group $\Gamma$. If each $P \in \mathbb{P}$ is finitely generated, the action on ${K}^o$ is relative to $\P$.
\end{proposition}

The proof of this result is given below. We note that some parts of the proof follow from arguments given in the proof of \cite[Thm.~8]{Swenson00}.
Before explaining the proof, we state a result of Levitt that will be used in the proof.

\begin{theorem}[\cite{Levitt98}, Thm.~3]
\label{thm:LevittThm3}
Let $T$ be a topological $\R$--tree with a non-nesting action of a group $G$.
   \begin{enumerate}
   \item
   \label{item:Levitt:Axis}
   If $g \in G$ acts on $T$ with no fixed point, then $g$ acts \emph{hyperbolically}, in the following sense.
   There is a proper embedding $\R \to T$ whose image $A$ is $g$--invariant and the action of $g$ on $A$ is topologically conjugate to a translation of\/ $\R$.  Such a set $A$ is an \emph{axis} for $g$.
   \item If $G$ is finitely generated and acts on $T$ with no global fixed point, then $G$ contains a hyperbolic element.
   \end{enumerate}
\end{theorem}

\begin{proof}[Proof of Proposition~\ref{prop:CutPointPretree}]
Let $\widehat{K}$ be the cut point tree of $M$ dual to $C$.
Let
\[
   R = \widehat{K} \setminus \{\text{terminal points} \}.
\]
The proof of \cite[Thm.~8]{Swenson00} involves several steps, one of which establishes the following. 
For any $\Gamma$--invariant subtree $S$ of $\widehat{K}$, let $S^o$ be the subtree of $S$ obtained by removing all terminal points.  Then the action of $\Gamma$ on $S^o$ is non-nesting and the stabilizer of each interval of $S^o$ is an elementary subgroup of the convergence action of $\Gamma$ on $M$.  In this portion of the proof, the hypothesis regarding infinite torsion subgroups is never used.  

Applying these conclusions to $\widehat{K}$ itself, we conclude that the action on the interior $K^o$ of the pretree $K$ is nonnesting and each arc stabilizer in $K^o$ is elementary. 
The action of $\Gamma$ on $K$ is without inversions by Corollary~\ref{cor:KWithoutInversion}.

Next we show that the action of $\Gamma$ on ${K}^o$ is relative to $\mathbb{P}$. Since the action on ${K}^o$ is without inversions, it suffices to show that the action on $R$ is relative to $\mathbb{P}$. Indeed, the action on the pretree extends to the intervals $I(x,y)$ such that if any point of $I(x,y)$ is fixed by a subgroup $P$ then $P$ fixes the closed interval $[x,y]$ pointwise; in particular it fixes $x$ and $y$ pointwise, which are members of the pretree.
It is clear that each $P \in \mathbb{P}$ has a fixed point in $\widehat{K}$. To complete the proof, we need to find such a fixed point that lies in the subtree $R$.

It is shown in \cite[Thm.~14]{PapasogluSwenson06} that the cut point tree $\widehat{K}$ has a canonical topology as a topological $\R$--tree.
Each $P \in \mathbb{P}$ is finitely generated with a non-nesting action on the topological $\R$--tree $R$.  If no point of $R$ is fixed by $P$, then by Theorem~\ref{thm:LevittThm3} some infinite order element $p\in P$ acts hyperbolically on ${R}$ with axis $L$; \emph{i.e.}, $p$ fixes no point in $R$, and there exists a proper embedding of the real line in $R$ whose image $L$ is $p$--invariant and the action of $p$ on $L$ is topologically conjugate to a translation.
The closure of $L$ in $\widehat{K}$ is an interval $[L_{-\infty}, L_{\infty}]$ whose endpoints are distinct ${K}$--equivalence classes of $M$ stabilized by $p$.  Neither endpoint is a cut point, since they are terminal in $\widehat{K}$, so there is a cut point $c \in (L_{-\infty},L_{\infty})$.
As $p(c)$ is another such cut point distinct from $c$, the sets $L_{-\infty}$ and $L_{\infty}$ have disjoint closures.
But $p$ is a parabolic element of the convergence group, so it has a unique fixed point in $M$, which is contained in any closed $p$--invariant subset \cite[\S 2]{Tukia94}, a contradiction.  We conclude that $P$ has a fixed point in $R$. 

To see that $\Gamma$ has no fixed point in $\widehat{K}$, we use the hypothesis that $\Gamma$ acts on $M$ minimally.
For each pair of open sets $U,V \subset M$ separated by a cut point $c \in C$, the density of loxodromic fixed point pairs \cite[Thm.~2R]{Tukia94} implies that there exists a loxodromic element $h \in \Gamma$ with one fixed point in $U$ and one in $V$.
Since its fixed points are separated by $c$, such an element has no fixed point in $R$ by Lemma~\ref{lem:SwensonLoxodromic}.
Thus $h$ acts hyperbolically on $R$ with an axis by Theorem~\ref{thm:LevittThm3}(\ref{item:Levitt:Axis}).
Since the orbit of $c$ is dense in $M$, there exists a pair of such loxodromic elements whose axes in $R$ do not share a common endpoint.
It follows that $\Gamma$ has no fixed point in $\widehat{K}$.
\end{proof}

The cut point tree $\widehat{K}$ has several useful dynamical properties.  However, in the next section we discuss a construction that leads to a dendrite $D$ with much stronger properties.  For instance, the action on $\widehat{K}$ has elementary interval stabilizers, but the action on $D$ has finite interval stabilizers (see Proposition~\ref{prop:NoSplitCutPoint}).  Lemma~\ref{lem: twoendedfix} discusses a critical collapsing property of $D$ that is not easily established for $\widehat{K}$.
Moreover, $D$ is a topological quotient of $M$, which is not the case for $\widehat{K}$.

\section{A dendrite quotient of the boundary}
\label{sec:  minimalcodense}

In \cite{Bowditch99_Treelike}, Bowditch constructs a natural dendrite quotient of any continuum dual to a chosen family of cut points.  If $M$ admits a convergence group action by a one-ended group $\Gamma$ and the chosen cut points are nonparabolic, then Bowditch shows that the dendrite quotient is nontrivial.
The main goal of this section is to prove Proposition~\ref{prop: reldendrite}, which extends Bowditch's result to the more general setting in which the group $\Gamma$ acts relative to $\P$ and does not split relative to $\mathbb{P}$ over a finite or parabolic subgroup.

\begin{definition}[Dendrite quotient]
\label{def:DendriteQuotient}
Let $M$ be a continuum and $C$ a set of cut points of $M$. For each pair $x,y \in M$, the set of cut points of $M$ separating $x$ from $y$ has a natural linear order.
The \emph{dendrite quotient of $M$ dual to $C$}, denoted $M / {\approx}$ or $D(M,C)$, is formed from $M$ by identifying two points $x,y$ if the collection of points of $C$ separating them does not contain a subset order-isomorphic to the rational numbers.
When endowed with the quotient topology, $D(M,C)$ is a dendrite, according to the main theorem of \cite{Bowditch99_Treelike}.
\end{definition}

In order to better understand the dendrite quotient, we examine the following transfinite sequence of quotient pretrees.

\begin{defn}[A transfinite sequence of quotients]
\label{def:Transfinite}
Let $M$ be a continuum with a set of cut points $C$. 
Let $K$ denote the cut point pretree of $M$ dual to $C$.
We recursively define a full equivalence relation $\sim_\alpha$ on $K$ with quotient $\pi_\alpha \colon K \to K_\alpha$ for each ordinal $\alpha$.
The pretree $K_0$ is equal to $K$ itself.
If $K_\alpha$ has been defined for some ordinal $\alpha$, we define $K_{\alpha+1}$ by identifying each pair of points $x,y\in K$ such that the interval between $\pi_\alpha(x)$ and $\pi_\alpha(y)$ in $K_\alpha$ is finite.
If $\alpha$ is a limit ordinal, we define $K_\alpha$ by identifying a pair $x,y \in K$ if $\pi_\beta(x)=\pi_\beta(y)$ for some $\beta<\alpha$.
\end{defn}

Bowditch shows that the recursive procedure described above eventually stabilizes, and the stable quotient may be described precisely, as follows.

\begin{lemma}[\cite{Bowditch99_Treelike}, Lem. 4.3 and~4.4]
\label{lem:BowditchLemma}
For any pretree $K$, the transfinite sequence of quotients defined as above eventually stabilizes.
In particular, there exists an ordinal $\alpha_0$ such that the natural map $K_{\alpha_0} \to K_{\alpha_0 + 1}$ is a bijection.
Furthermore, the stable quotient $K_{\alpha_0}$ is equal to the quotient $K/{\approx}$ such that $x\approx y$ if the pretree interval $(x,y)$ in $K$ does not contain a subset order-isomorphic to the rational numbers.
\end{lemma}

\begin{proposition}
\label{prop:DendritePretree}
Let $M$ be a continuum, let $C$ be a set of cut points of $M$, and let $K$ be the cut point pretree.
The quotient map $M \to K$ induces a bijection from the dendrite $D(M,C) = M/{\approx}$ to the pretree $K_{\alpha_0}$ for some ordinal $\alpha_0$.
This bijection is also a pretree automorphism.
\end{proposition}

\begin{proof}
According to Lemma~\ref{lem:BowditchLemma}, the stable quotient $K_{\alpha_0}$ is equal to the quotient $K/{\approx}$.
By Definition~\ref{def:CutpointPretree}, the members of $K$ have two types, singleton sets $\{c\}$ with $c \in C$ and equivalence classes of points of $M$ not contained in $C$ under the relation $a\mathcal{R} b$ if $a$ and $b$ are separated in $M$ by no point of $C$.
Any two equivalence classes of the second type are separated by a cut point from $C$, by definition.
If an interval of $K$ contains a subset $R$ order-isomorphic to the rationals, then $R\cap C$ is also isomorphic to the rationals.
Thus, $x \approx y$ holds if and only if the members of $C$ between $x$ and $y$ do not contain a subset order isomorphic to the rational numbers.

The relation $\approx$ on $K$ pulls back to the relation $\approx$ on $M$ such that $x\approx y$ if the set of points of $C$ separating $x$ from $y$ does not contain a subset order isomorphic to the rationals.  Therefore, the map $M\to K$ induces a pretree automorphism from the dendrite $D(M,C)$ to the stable pretree $K_{\alpha_0}$. 
\end{proof}

The two constructions of the dendrite $D(M,C)$ described above may seem a bit arbitrary.  But in fact, the dendrite quotient operation is natural from a categorical point of view, as described in the following proposition.
This proposition is not used in any of our main arguments, and it may safely be skipped.

\begin{proposition}
\label{prop:MaximalHausdorff}
Let $D(M,C)$ be the dendrite quotient of $M$ dual to $C$.
Let $\mathcal{R}$ be the relation on $M$ such that $x\mathcal{R} y$ holds if $x$ and $y$ are points of $M\setminus C$ that are not separated by any point of $C$.
If $M/\mathcal{R}$ is endowed with the quotient topology, then $D(M,C)$ is equal to the maximal Hausdorff quotient of $M/\mathcal{R}$.
\end{proposition}

The \emph{maximal Hausdorff quotient} of a space $X$ is the universal Hausdorff quotient $HX$ of $X$ such that every map from $X$ to a Hausdorff space factors through $HX$ (see \cite[Prop.~V.9.2]{MacLane_Categories} or \cite[Lem.~3.1]{McDuff06}).
The proof of Proposition~\ref{prop:MaximalHausdorff} uses the notion of connector points and the following lemma, which will be used several more times in our main results.

\begin{lemma}
\label{lem:UniqueConnector}
Let $X$ and $Y$ be adjacent points of $K_\alpha$ for $\alpha\ge 1$.
Then one of the pairs $(X,Y)$ and $(Y,X)$ has a unique connector point and the other pair has no connector point.
\end{lemma}

\begin{proof}
Recall that $K$ is complete and median by Definition~\ref{def:CutpointPretree}.
When considered as subsets of $K$, the sets $X$ and $Y$ are nonempty, disjoint,  and full with $X \cup Y$ also full. 
Under those circumstances, Lemma~\ref{lem:Connector} shows that at least one of the pairs has a unique connector point.

If the pairs $(X,Y)$ and $(Y,X)$ both had connector points, then the two connectors $p$ and $p'$ would be adjacent in ${K}$ and in the same $\sim_1$--class, contradicting the fact that $p$ and $p'$ are not $\sim_\alpha$--equivalent.
\end{proof}

\begin{proof}[Proof of Proposition~\ref{prop:MaximalHausdorff}]
As the dendrite $D(M,C)$ is Hausdorff and equal to $K_{\alpha_0}$ for some ordinal $\alpha_0$, it suffices to show that any continuous map from $M/\mathcal{R}$ to a Hausdorff space factors through the quotient $M/\mathcal{R} = K \to K_\alpha$.

We proceed by induction on $\alpha$.
The claim is trivial for $\alpha=0$ since $K=K_0$.
Recall that $K_1$ is formed from $K$ by identifying adjacent pairs of points $X,Y$ in $K$.  If $X$ and $Y$ are adjacent in $K$, then---considered as subsets of $M$---one is contained in the closure of the other by Lemma~\ref{lem:KAdjacent}(\ref{item:KAdjacent:Closure}). Thus $X$ and $Y$ have the same image in any Hausdorff space, and any map from $M/\mathcal{R}$ to a Hausdorff space factors through $K_1$.
The inductive claim is immediate for limit ordinals, so we only need to consider the successor case. 

Suppose any map from $M/\mathcal{R}$ to a Hausdorff space factors through $K_\beta$ for some $\beta\ge 1$, and let $\alpha=\beta+1$.  Then $K_\alpha$ is formed from $K_\beta$ by identifying all adjacent pairs of points in $K_\beta$.
Let $X$ and $Y$ be adjacent elements of $K_\beta$, which we consider as full disjoint subsets of $K$ whose union is full.
Without loss of generality, assume that the connector point $p = p(Y,X)$ exists.
Select any $x \in X$, and let $\hat{C}$ be the set of all $c \in C$ separating $x$ from $p$.
Then $\{c\} \in X$ for each $c\in \hat{C}$, since $[x,p]\cap Y = \{p\}$.
Furthermore, no point of $X$ is adjacent to $p$ by Lemma~\ref{lem:UniqueConnector}.

We claim that $\hat{C}$ has a limit point in common with the set $p$.
Consider the points $x,p\in K$ as subsets of $M$, and choose $x_0\in x$ and $p_0\in p$.
Let $I$ be an irreducible subcontinuum of $M$ from $x_0$ to $p_0$,
\emph{i.e}, a subcontinuum of $M$ containing $x_0$ and $p_0$ such that no proper subcontinuum of $I$ contains both $x_0$ and $p_0$.
By irreducibility, the points of $\hat{C}$ are precisely the members of $C$ that separate distinct points of $I$.
For each $c \in \hat{C}$, let $I_c$ be the subcontinuum of $I$ consisting of $c$ together with the component of $I\setminus\{c\}$ containing $p_0$.
Then $I_{c'} \subseteq I_c\setminus \{c\}$, whenever $c' \in \hat{C}$ lies between $c$ and $p_0$. 
If we let $I_\omega$ be the intersection of all $I_c$ for $c \in \hat{C}$, then $I_\omega$ has empty intersection with $\hat{C}$ since for each $c \in \hat{C}$ there exists $c'\in \hat{C}$ between $c$ and $p_0$.

By compactness of $M$, any neighborhood $U$ of $I_\omega$ contains $I_c$ for some $c \in\hat{C}$.
Thus the closure of $\hat{C}$ intersects $I_\omega$.
In order to establish the claim,
it suffices to show that $I_\omega$ is contained in the closure of $p$.
If $p_0 \notin C$, the intersection $I_\omega$ is a subcontinuum of $I$ not separated by any member of $C$. 
By Lemma~\ref{lem:KAdjacent}(\ref{item:KAdjacent:Cyclic}), the set $I_\omega \setminus C$ is dense in $I_\omega$ and contained in the $\mathcal{R}$--equivalence class $p$. We conclude that $I_\omega$ is in the closure of $p$.
On the other hand, suppose $p_0 \in C$.
If $I_\omega$ contained more than the one point $p_0$, then $I_\omega\setminus C$ would be a single $\mathcal{R}$--equivalence class contained in $X$ and adjacent to $\{p_0\}$, contradicting Lemma~\ref{lem:UniqueConnector}.
Thus we must have $I_\omega = \{p_0\} = p$.

Since $\hat{C}$ has a limit point in common with $p$, the closures of $X$ and $Y$ intersect in $K$.
In particular, $X$ and $Y$ must have the same image in any Hausdorff space, and any map from $M/\mathcal{R}$ to a Hausdorff space factors through $K_\alpha$.
\end{proof}

Throughout the rest of this section we work under the following assumptions. Let $M$ denote a continuum with a minimal convergence group action by a finitely generated group $\Gamma$ relative to a family of subgroups $\P$.
We assume that $\Gamma$ does not split relative to $\P$ over any finite or parabolic subgroup.
Assume $M$ has a nonparabolic cut point $c$ with orbit $C = \Gamma(c)$.
We let ${K}$ be the cut point pretree dual to $C$, and let $(K_\alpha)$ be the transfinite sequence of quotients of $K$ given by Definition~\ref{def:Transfinite}.

\begin{lemma}
\label{lem: relative in quotient}
For each ordinal $\alpha$, the action of\/ $\Gamma$ on the interior of $K_\alpha$ is non-nesting, without inversions, and relative to $\P$.
\end{lemma}

\begin{proof}
By Proposition~\ref{prop:CutPointPretree}, the action of $\Gamma$ on ${K}^o$ is non-nesting and relative to $\P$.
Since $K$ is complete and median, its interior is interval complete and median.
Let $K'$ be the $\Aut(K)$--invariant full subpretree of $K$ given by Lemma~\ref{lem: interior parabolic point of pretree}.
Since each $P \in \P$ has a fixed point $x$ in $K^o$, it also has a fixed point $x'$ in $K'$ by Lemma~\ref{lem:ParabolicInterior}.
Since $\sim_\alpha$ is a full relation on $K$, the image $[x']$ in $K_\alpha$ lies in the interior of $K_\alpha$ and is also fixed by $P$.
Similarly, since the action of $\Gamma$ on $K'$ is non-nesting, the induced action on the interior of $K_\alpha$ is non-nesting by Lemma~\ref{lem: non-nesting in quotient}.
By Corollary~\ref{cor:KWithoutInversion}, the action on $K=K_0$ is without inversions.
Lemma~\ref{lem:UniqueConnector} implies the same conclusion for $\alpha \ge 1$.
\end{proof}

\begin{remark}[Collapsing simplicial trees]
\label{rem:Collapsing}
Note that, in the successor case, $K_{\beta+1}$ may naturally be considered as a quotient of $K_\beta$ by the relation that two points $x,y\in K_\beta$ are equivalent if only finitely many points of $K_\beta$ lie between $x$ and $y$.
Each equivalence class in $K_\beta$ is a discrete, full, subpretree of $K_\beta$, which may be considered as the vertex set of a simplicial tree (see \cite[Lem.~3.34]{Bowditch99_Treelike}).
Two vertices are joined by an edge if and only if they are adjacent in $K_\beta$.
\end{remark}

According to Proposition~\ref{prop:CutPointPretree}, the action of $\Gamma$ on $K_0=K$ has no fixed point.  We now consider $K_1$.
The following lemma depends on our assumption that the family $C$ of cut points of $M$ contains no parabolic points.

\begin{lemma}
\label{lem:K1Nontrivial}
The action of\/ $\Gamma$ on $K_1$ does not have a fixed point.
\end{lemma}

\begin{proof}
Suppose $\Gamma$ has a fixed point $X$ in $K_1$.
Since $K_1$ is the quotient of $K$ by the finite interval relation, the equivalence class $X$ is a $\Gamma$--invariant discrete, full subpretree of $K$ that may be identified with the set of vertices of a simplicial tree (see Remark~\ref{rem:Collapsing}).
Observe that $X$ must contain at least three distinct points.  Indeed, as $\Gamma$ has no fixed point, $X$ contains more than one point. If $X$ had exactly two points $a$ and $b$, then they would be adjacent, since $X$ is full. But $\Gamma$ acts without inversions, so it would fix each of $a$ and $b$, which is impossible.
Since $X$ has three distinct points, $a$, $b$, and $c$, it also contains their median, which lies in the interior $K^o$ of $K$.

We now consider the nonempty subpretree $X^o$ obtained from $X$ by removing all terminal points, and we consider $X^o$ to be the set of vertices of a simplicial tree $T$. The action of $\Gamma$ on $K^o$ is non-nesting, relative to $\P$, and without inversions by Proposition~\ref{prop:CutPointPretree}.  Thus, the action of $\Gamma$ on $T$ is relative to $\P$ by Lemma~\ref{lem:ParabolicInterior} and is also clearly without inversions.

By the construction of $K$, the tree $T$ is bipartite with bipartition $V(T) = V_0(T) \sqcup V_1(T)$.
The set $V_0(T)$ consists of members of $K$ that are not contained in $C$, and the set $V_1(T)$ consists of singleton sets $\{c\}$ with $c \in C$.
Since the points of $C$ are nonparabolic in the convergence group action on $M$, the stabilizer of a vertex $v\in V_1(T)$ is either finite or loxodromic in the convergence group $\Gamma$.

We claim that a vertex $v\in V_1(T)$ is incident to at most one edge of $T$ with loxodromic stabilizer. Suppose, by way of contradiction, that edges $e,e'$ are adjacent to $v$ and have loxodromic stabilizers $\Gamma_e$ and $\Gamma_{e'}$.
Since these loxodromic subgroups share a fixed point, they must be co-elementary. Thus there is a loxodromic element $g \in \Gamma$ that stabilizes both $e$ and $e'$.
Suppose $v=\{c\}$ for a cut point $c\in C$.
Then there exist $w,w'\in V_0(T)$ members of $K$ adjacent to $\{c\}$ such that $g$ stabilizes $w$ and $w'$ and also fixes $c$.
Let $y \in M$ be the other fixed point of $g$.  By Lemma~\ref{lem:SwensonLoxodromicfollowup}, every member of $K$ stabilized by $g$ lies in the closed interval $\bigl[ [c],[y] \bigr]$.
But $w$ and $w'$ are members of $K$ fixed by $g$ and $c$ is between them, which is a contradiction.

Consider the tree $\bar{T}$ formed from $T$ by collapsing to a point each edge whose stabilizer is loxodromic. Then the preimage of each vertex under the collapse $T \to \bar{T}$ is a subtree of diameter at most two, since it contains at most one vertex from $V_0(T)$.  In particular, $T \to \bar{T}$ is an equivariant quasi-isometry, so any subgroup of $\Gamma$ fixing a vertex of $\bar{T}$ also fixes a vertex of $T$.
But $\Gamma$ acts on $\bar{T}$ without inversions, relative to $\P$, and with only finite edge stabilizers, so $\Gamma$ fixes a vertex of $\bar{T}$ by our hypothesis that $\Gamma$ does not split relative to $\P$ over finite subgroups.
Hence, $\Gamma$ also fixes a vertex of $T$, contradicting the fact that $\Gamma$ has no fixed point in $K$.
\end{proof}

The following two lemmas are proved using arguments inspired by the proof of \cite[Thm.~6.1]{Bowditch99_Treelike}.

\begin{lemma}
\label{lem: edge stabs are finite}
Let $X$ and $Y$ be two adjacent points in the interior ${K}^o_{\alpha}$ of ${K}_\alpha$ for any ordinal $\alpha\ge 1$. Then the stabilizer of\/ $\{X, Y\}$ is either finite or a parabolic subgroup of the convergence action on $M$.
\end{lemma}

\begin{proof}
In a convergence group, any subgroup with no loxodromic elements must be either finite or parabolic.  We prove the lemma in contrapositive form, as follows.
Suppose the stabilizer of $\{X,Y\}$ contains an infinite order element $\gamma$ that acts loxodromically on $M$.
We show that either $X$ or $Y$ must be terminal in the pretree $K_\alpha$.

By Lemma~\ref{lem:UniqueConnector}, one of the pairs, say $(Y,X)$, has a unique connector point $p=p(Y,X)$ and the other pair $(X,Y)$ has no connector point.
Recall that the points $X$ and $Y$ of $K_\alpha = K/\!\!\sim_{\alpha}$ are subsets of $K$, and $p$ is an element of $K$ that is also an element of the set $Y$.
Since $\gamma$ stabilizes $\{X,Y\}$, the point $p$ must be fixed by $\gamma$.

If $x,y \in X$ and we direct the intervals $[x,p]$ and $[y,p]$ so that $x<p$ and $y<p$,
then the half-open intervals $[x,p)$ and $[y,p)$ are cofinal, since the median of $\{x,y,p\}$ cannot equal $p$.
We claim that $X$ contains a point of $K$ that is fixed by $\gamma$.
Suppose, by way of contradiction, that $\gamma$ does not fix any point of $X$.
Choose any $y \in X$ and choose $x \in [y,p) \cap \bigl[ \gamma(y),p \bigr) \subseteq X$.
Now, $\gamma(x) \in \gamma\bigl( y,p) \bigr) = \bigl[ \gamma(y),p \bigr)$.
Since $x$ and $\gamma(x)$ are, by assumption, distinct, and both lie in $\bigl[ \gamma(y),p \bigr)$, we must have either $x \in \bigl( p,\gamma(x) \bigr)$ or $\gamma(x) \in (p,x)$.
As $p$, $x$, and $\gamma(x)$ lie in the interior $K^o$, this arrangement violates the non-nesting property of $K^o$, established in Proposition~\ref{prop:CutPointPretree}.

Let $q \in X$ be a fixed point of $\gamma$. Since $q \notin Y$, we see that $p \not\sim_\alpha q$.  It follows that $p \not\sim_1 q$. In other words, in the pretree $K$, the interval $(p,q)$ is infinite.
Since $\gamma$ is loxodromic, we must be in the first case of Lemma~\ref{lem:SwensonLoxodromic}. Thus the classes of $p$ and $q$ in $K$ are terminal, contradicting the assumption that $X$ and $Y$ are non-terminal.
\end{proof}

\begin{lemma}
\label{lem: pretree quotient is non trivial relative to P and nonnesting}
For each $\alpha$, the action of\/ $\Gamma$ on ${K}_\alpha$ does not have a fixed point.
The action of\/ $\Gamma$ on the interior ${K}_\alpha^o$ is non-nesting and relative to $\P$.
\end{lemma}

\begin{proof}
It follows directly from Lemma~\ref{lem: relative in quotient} that the action on the interior of $K_\alpha$ is non-nesting, relative to $\P$, and without inversions for each $\alpha$.
We prove by transfinite induction on $\alpha$ that $\Gamma$ acts on $K_\alpha$ without a fixed point for all $\alpha$.
Observe that nontriviality is known for $K_1$ by Lemma~\ref{lem:K1Nontrivial}.

Let $\alpha$ be a limit ordinal.
Assume that $\Gamma$ is known to act nontrivially on $K_\beta$ for all $\beta<\alpha$.
Suppose, by way of contradiction that $\Gamma$ fixes a point $X$ in $K_\alpha$. 
Then $X$ is an equivalence class in ${K}$ stabilized by $\Gamma$. Choose any $x \in X$ and a finite generating set $\{s_1,\dots, s_n\}$ of $\Gamma$. Clearly as $s_i (x) \sim_{\alpha} x$, there exists $\beta_i < \alpha$ such that $x \sim_{\beta_i} s_i(x)$ for each $i$. Taking $\beta =\max\{\beta_1,\dots,\beta_n\}$, we see that $\Gamma$ fixes a point in ${K}_{\beta}$, contradicting our hypothesis since $\beta<\alpha$.

For the successor step, suppose $\Gamma$ is known to act nontrivially on $K_\beta$ for some $\beta\ge 1$.  We show that $\Gamma$ acts nontrivially on $K_\alpha$ for $\alpha=\beta+1$.
By way of contradiction, suppose $\Gamma$ fixes a point $x$ in $K_\alpha$. 
As explained in Remark~\ref{rem:Collapsing}, since $K_\alpha$ is the quotient of ${K}_{\beta}$ by the finite interval relation, the preimage of $\{x\}$ in $K_\beta$ is a $\Gamma$--invariant, discrete, full subpretree $X$.
As in the proof of Lemma~\ref{lem:K1Nontrivial}, the interior $X^o$ is a nonempty discrete full subpretree of $K^o_\beta$ on which $\Gamma$ acts relative to $\P$ without inversions.

The discrete median pretree $X^o$ may be considered as the set of vertices of a simplicial tree $T$
on which $\Gamma$ acts without inversions relative to $\P$.
Moreover, by Lemma~\ref{lem: edge stabs are finite} the stabilizer of each edge of $T$ is either finite or parabolic.
As we are working under the hypothesis that $\Gamma$ does not split over such a subgroup relative to $\P$, there is a vertex of $T$ fixed by $\Gamma$, contradicting the inductive hypothesis that $\Gamma$ has no fixed point in $K_\beta$.
\end{proof}

Combining the results established above gives the following conclusion.

\begin{proposition}
\label{prop: reldendrite}
Let $(\Gamma,\P)$ be relatively hyperbolic with boundary $M = \boundary(\Gamma,\P)$.
Assume that $\Gamma$ does not split relative to $\P$ over a finite or parabolic subgroup.
If $M$ has a nonparabolic cut point $c$ with orbit $C = \Gamma(c)$, then the induced action of\/ $\Gamma$ on the dendrite quotient $D(M,C)$ is a convergence group action without a fixed point.
The action on the interior of $D(M,C)$ is non-nesting and relative to $\P$ with finite interval stabilizers.
\end{proposition}

We now use the fact that the quotient $D(M,C)$ of $M$ is Hausdorff.

\begin{proof}
We assume $(\Gamma,\P)$ is nonelementary, since otherwise the result is trivial.
Thus the action of $\Gamma$ on the compactum $M$ is a minimal convergence group action.
Since $\Gamma$ does not split over a finite or parabolic subgroup relative to $\P$, the group $\Gamma$ is finitely generated (Theorem~\ref{thm:OsinFinGen}) and the space $M$ is connected (Proposition~\ref{prop: stallings}) and metrizable (\cite{Gerasimov09}).

As explained in Proposition~\ref{prop:DendritePretree}, the dendrite quotient $D=D(M,C)$ is isomorphic as a pretree to the stable pretree $K_{\alpha_0}$.
Therefore, the action of $\Gamma$ on $D$ is nontrivial, and the action on its interior is non-nesting and relative to $\P$ by Lemma~\ref{lem: pretree quotient is non trivial relative to P and nonnesting}.
Since $D$ is Hausdorff, the quotient $M\to D$ is a closed map. As $\Gamma$ has a minimal convergence group action on $M$, the induced action on $D$ is also a minimal convergence group action.

Let $I=[x,y]$ be a nondegenerate interval in the interior of $D$. The stabilizer of $I$ has a subgroup $H$ of index at most two fixing $x$ and $y$.  By non-nesting, $H$ fixes $I$ pointwise.  As $H$ is a convergence group on $M$ with at least three fixed points, $H$ is finite, so that $\Stab(I)$ is finite as well.
\end{proof}

\section{From topological trees to metric and simplicial trees}
\label{sec:Isometric action of a relatively hyperbolic group}

In this section, we extend a theorem of Levitt from the setting of finitely presented groups to groups with a finite relative presentation. The main result of the section is Corollary~\ref{cor: relativerips thm}, which produces an action on a simplicial tree from a non-nesting action on a topological $\R$--tree.

\begin{definition}
Given a group pair $(\Gamma,\P)$, we say that $\Gamma$ is \emph{finitely presented relative to $\P=\{P_1,\dots,P_k\}$} if there exists a finitely generated free group $F$ such that $\Gamma$ is the quotient of $P_1 * \cdots * P_k * F$ by the normal closure of a finite subset.
\end{definition}

The following result was established for finitely presented groups by Levitt \cite{Levitt98}.  We extend that result to certain relatively finitely presented groups by adapting a technique used in the proof of \cite[Thm.~9.9]{GuirardelLevitt_Splittings}.

\begin{proposition}
\label{prop:Levitt}
Suppose the finitely generated group $\Gamma$ is also finitely presented relative to a finite family of finitely generated subgroups $\mathbb{P}$.
If\/ $\Gamma$ admits a nontrivial non-nesting action on a topological $\R$--tree $R$ relative to $\mathbb{P}$, then it admits a nontrivial isometric action on some $\R$--tree $R_0$ relative to $\mathbb{P}$ such that every subgroup fixing an interval in $R_0$ also fixes an interval in $R$. 
\end{proposition}

\begin{proof}
The first step of the proof of \cite[Thm.~1]{Levitt98} is to construct a certain resolution of the action on $R$ whenever $\Gamma$ is finitely presented.  As in \cite{GuirardelLevitt_Splittings}, we show that such a resolution exists under the weaker hypothesis that $(\Gamma,\P)$ has a finite relative presentation and $\Gamma$ and each $P\in \P$ are finitely generated.
For each $P \in \P$, choose a presentation $P = \presentation{S_P}{\mathcal{W}_P}$ with $S_P$ finite.
Fix a presentation $\Gamma = \presentation{S}{\mathcal{W}}$ such that $S$ is a finite set containing each $S_P$ and such that $\mathcal{W}$ contains each $\mathcal{W}_P$ together with a finite set $\mathcal{W}_0$ of additional relators.

Let $K$ be a subtree of $R$ that is the convex hull of a finite set of points.
For each $s \in S$, let $K_s = K \cap s^{-1}K$ and consider the restriction $\phi_s\colon K_s \to sK_s$ of $s$.
Choose $K$ large enough that each set $K_s$ contains more than one point, that for each relator $r=s_1^{\epsilon_1} \cdots s_n^{\epsilon_n}$ in $\mathcal{W}_0$ the domain of the map $\phi_{s_1}^{\epsilon_1} \cdots \phi_{s_n}^{\epsilon_n}$ is nonempty, and that $K$ contains a point $x_P$ fixed by $P$ for each $P \in \P$.

Build a foliated $2$--complex $\Sigma$ by starting with $K$ (foliated by points) and gluing bands $K_s \times [0,1]$ (foliated by $\{*\}\times [0,1]$) to $K$ so that $(x,0)$ is glued to $x$ and $(x,1)$ is glued to $\phi_s(x)$.  
Then $\pi_1(\Sigma)$ may be naturally identified with the free group on $S$.
By construction, each relator of $\mathcal{W}$ is represented (up to free homotopy) by a loop contained in a leaf of $\Sigma$.
The natural homomorphism $\rho\colon \pi_1(\Sigma) \to \Gamma$ gives rise to a foliated covering space $p \colon \bar\Sigma_\rho \to \Sigma$.
Furthermore, there exists a unique $\Gamma$--equivariant \emph{resolution map} $f\colon \bar\Sigma_\rho \to R$ that restricts to the covering map $p$ on some fixed component of $p^{-1}(K)$ and is constant on each leaf of the foliation.

The rest of the argument of \cite{Levitt98} follows without change using this resolution map.
The finite system of maps $\{\phi_s\}$ is closed and non-nesting in the sense of \cite{Levitt98} since the action on $R$ is non-nesting.
Given such a system, Levitt shows that $K$ admits a probability measure $\mu$ with no atom such that each $\phi_s$ maps the restriction $\mu|K_s$ onto $\mu|sK_s$.
The measure on $K$ induces a transverse measure (perhaps not of full support) on the foliated cover $\bar\Sigma_\rho$.  Integration of this measure along paths gives a length pseudometric on $\bar\Sigma_\rho$. Identifying points $u,v$ with $d(u,v)=0$ produces a metric $\R$--tree $R_0$ on which $\Gamma$ acts nontrivially and isometrically relative to $\P$ such that every subgroup fixing an interval in $R_0$ also fixes an interval in $R$. 
\end{proof}

The fundamental structure theorem for stable actions of finitely presented groups on $\R$--trees is the splitting theorem of Bestvina--Feighn \cite{BestvinaFeighn_RipsMachine}.
A theorem of Osin \cite{Osin06_RelHyp} states that every relatively hyperbolic pair $(\Gamma,\P)$ is relatively finitely presented, and if $\Gamma$ is finitely generated then so is each $P \in \P$.
A relatively hyperbolic analogue of the Bestvina--Feighn structure theorem due to Guirardel--Levitt \cite{GuirardelLevitt_Splittings} gives the following corollary.

\begin{corollary}
\label{cor: relativerips thm}
Let $(\Gamma,\mathbb{P})$ be relatively hyperbolic with $\Gamma$ finitely generated. 
If\/ $\Gamma$ acts nontrivially by non-nesting homeomorphisms on a topological $\R$--tree $R$ relative to $\P$ with elementary arc stabilizers, then $\Gamma$ splits over an elementary subgroup relative to $\P$. \qed
\end{corollary}

\section{Cut points and local connectedness}
\label{sec:MainThm}

This section contains the proof of the main theorem using the ingredients established above.  Broadly speaking, we follow a strategy analogous to that of Swarup and Bowditch \cite{Swarup96,Bowditch99_Boundaries,Bowditch99_Connectedness,Bowditch01_Peripheral}.  
Various components of Swarup and Bowditch's proofs require significant extension to apply in the general relatively hyperbolic setting.
In this section, we prove Proposition~\ref{prop:Accessibility}, a key accessibility result for splittings over two-ended groups.
Many other such extensions are established in the sections above.
Once we have these extensions in place, the local connectedness proof follows the overall plan of the original proofs due to Swarup and Bowditch.

If $(\Gamma,\P)$ is relatively hyperbolic, let $\mathbb{L}$ be the family of all loxodromic subgroups, \emph{i.e.}, subgroups that are virtually infinite cyclic and not parabolic.

\begin{proposition}
\label{prop:Accessibility}
Let $(\Gamma,\P)$ be relatively hyperbolic such that $\Gamma$ is one ended relative to $\P$.
   \begin{enumerate}
   \item Let $T$ be the tree of cylinders associated to any $(\mathbb{L},\P)$--tree.  Then each vertex stabilizer $\Gamma_v$ is hyperbolic relative to $\Q_v = \Inc_v \cup \P_{|\Gamma_v}$, as defined in Proposition~\ref{prop:RelQCVertex}.
   \item Suppose $\Gamma$ does not split over parabolic subgroups relative to $\P$.
   Then there exists an $(\mathbb{L},\P)$--tree $T_\ell$ that is equal to its tree of cylinders, such that each vertex stabilizer $\Gamma_v$ is finitely generated and does not split over loxodromic subgroups relative to $\Q_v$.
   \end{enumerate}
\end{proposition}

\begin{proof}
The first assertion is given by Proposition~\ref{prop:RelQCVertex}.
We now consider the second assertion.
We may assume $\Gamma$ is nonelementary, since the result is trivial in the elementary case.
Let $\mathbb{E}$ be the family of all elementary subgroups of $(\Gamma,\P)$. By hypothesis, any $(\mathbb{E},\P)$--tree is actually an $(\mathbb{L},\P)$--tree.
Furthermore, $\Gamma$ and all members of $\P$ are finitely generated by Theorem~\ref{thm:OsinFinGen}.
The JSJ tree of cylinders over $\mathbb{E}$ relative to $\P$ given by Guirardel--Levitt \cite[Cor.~9.20]{GuirardelLevitt_JSJ} is an $(\mathbb{L},\P)$--tree $T$ with the following property.
If $v$ is any vertex of $T$, either $\Gamma_v$ has no splittings over loxodromic subgroups relative to $\Q_v$ or $\Gamma_v$ is an extension with finite kernel of the fundamental group of a compact hyperbolic $2$--orbifold $\Sigma_v$ with geodesic boundary and the subgroups $\Q_v$ are the lifts to $\Gamma_v$ of the fundamental groups of the boundary components of $\Sigma_v$.
In the orbifold case, every loxodromic splitting of $\Gamma_v$ relative to $\Q_v$ is induced by a splitting of $\pi_1(\Sigma)$ dual to a family of disjoint, essential simple closed geodesics of $\Sigma_v$ by Morgan--Shalen \cite[Thm.~III.2.6]{MorganShalen84}.
A standard argument involving reduction of Euler characteristic shows that $\Sigma_v$ may be cut along some finite family of geodesics into pieces containing no essential simple closed geodesic.
Thus, we may refine $T$ using splittings of its vertex groups to produce an $(\mathbb{L},\P)$--tree $T_\ell$ each of whose vertex groups $\Gamma_v$ does not split relative to $\Q_v$.

Finally we observe that each vertex stabilizer $\Gamma_v$ of $T_\ell$ is finitely generated.
Indeed, $\Gamma_v$ is hyperbolic relative to $\mathbb{Q}_v$, but each member of $\P_{|\Gamma_v}$ is conjugate in $\Gamma$ to a member of $\P$ and each member of $\Inc_v$ is virtually cyclic.
In particular, $\Gamma_v$ is hyperbolic relative to finitely generated subgroups, which implies that $\Gamma_v$ is finitely generated (see \cite{Osin06_RelHyp,GerasimovPotyagailo15}).
\end{proof}

For the next two lemmas, we work under the following assumptions.
Let $(\Gamma,\P)$ be relatively hyperbolic.
Assume $\Gamma$ does not split relative to $\P$ over a finite or parabolic subgroup.
We note that, by Proposition~\ref{prop: stallings} the boundary $M=\boundary(\Gamma,\P)$ is connected.
Furthermore, by Theorem~\ref{thm:OsinFinGen} the group $\Gamma$ is finitely generated.
Suppose the boundary $M$ has a nonparabolic cut point $c$ with orbit $C=\Gamma (c)$ in $M$.
Let $R$ be the topological $\R$--tree obtained by removing all terminal points from the dendrite $D(M,C)$.

\begin{lemma}
\label{lem: twoendedfix}
Suppose $\Gamma$ splits over a loxodromic subgroup $H$ relative to $\P$.
Then the induced action of $H$ on the topological $\R$--tree $R$ has a fixed point.
\end{lemma}

\begin{proof}
Recall that the dendrite $D$ is formed from $M$ by identifying two points $x$ and $y$ if the collection of points of $C$ separating them does not contain a subset order-isomorphic to the rational numbers.
The action of $\Gamma$ on $M$ is a convergence group action, and, by Proposition~\ref{prop: reldendrite}, the induced action of $\Gamma$ on $D$ is also a convergence group action.

By a result of Haulmark--Hruska about splittings of relatively hyperbolic groups over elementary subgroups \cite[Cor.~6.8]{HaulmarkHruska_Canonical}, the two-point limit set $\Lambda H$ is not separated by any cut point of $M$.
In particular, the quotient map $M\to D$ sends $\Lambda H$ to a single point $z$ fixed by $\Gamma_v$.
Let $h \in H$ be an infinite order element.
If $h$ had another fixed point $w\in D$, its preimage in $M$ would be a closed $h$--invariant set not intersecting the limit set of $\langle h \rangle$ in $M$, which is impossible in a convergence group action (see \cite[\S 2]{Tukia94}). Here, we use that $D$ is a $T_1$--space, so that the preimage of $w$ in $M$ is closed.
Thus $h$ has $z$ as its unique fixed point in $D$.

We claim that $z$ lies in the interior $R$.
Recall that the action of $\Gamma$ on $R$ is non-nesting by 
Proposition~\ref{prop: reldendrite}.
As in the proof of Proposition~\ref{prop:CutPointPretree}, if the fixed point $z$ of $h$ does not lie in $R$, then \cite[Thm.~3]{Levitt98} implies that $h$ fixes a pair of terminal points of $D$ that are the endpoints of an axis.
Since $z$ is the unique fixed point of $h$ in $D$, we conclude that $z$ must lie in $R$.
\end{proof}

\begin{lemma}
\label{lem:RigidFix}
Let $T_\ell$ be an $(\mathbb{L},\P)$--tree
as in Proposition~\ref{prop:Accessibility} such that each vertex group $\Gamma_v$ does not split over loxodromic subgroups relative to $\Q_v$.
Then the induced action of each $\Gamma_v$ on the topological $\R$--tree $R$ has a fixed point.
\end{lemma}

\begin{proof}
Recall that $\Gamma_v$ is finitely generated and hyperbolic relative to $\Q_v = \Inc_v \cup \P_{|\Gamma_v}$.
By Proposition~\ref{prop: reldendrite}, the action of $\Gamma$ on $R$ is relative to $\P$, so that the restricted action of $\Gamma_v$ on $R$ is relative to $\P_{|\Gamma_v}$.
The subgroups $\Inc_v$ are also finitely generated, since they are virtually cyclic.
By Lemma~\ref{lem: twoendedfix}, the action on $R$ is also relative to the family of loxodromic edge stabilizers $\Inc_v$.

If the action of $\Gamma_v$ on $R$ had no fixed point, it would be a nontrivial non-nesting action relative to $\Q_v$ such that every interval stabilized by $\Gamma_v$ has a finite stabilizer by Proposition~\ref{prop: reldendrite}.
But then $\Gamma_v$ would split over an elementary subgroup relative to $\Q_v$ by Corollary~\ref{cor: relativerips thm}.
Any splitting of $\Gamma_v$ relative to $\Q_v$ over a finite or parabolic subgroup extends to a splitting of $\Gamma$ relative to $\P$ over finite or parabolic subgroups, so $\Gamma_v$ has no such splitting.
Thus $\Gamma_v$ must split over a loxodromic subgroup relative to $\Q_v$, contradicting the choice of $T_\ell$.
\end{proof}

\begin{proposition}
\label{prop:NoSplitCutPoint}
Let $(\Gamma,\P)$ be relatively hyperbolic.
Suppose $\Gamma$ does not split over a finite or parabolic subgroup relative to $\P$.
Then $M=\boundary(\Gamma,\P)$ is connected and every cut point of $M$ is a parabolic point.
\end{proposition}

\begin{proof}
By Proposition~\ref{prop: stallings} and Theorem~\ref{thm:OsinFinGen}, the space $M$ is connected and the group $\Gamma$ is finitely generated.
We suppose, by way of contradiction that $M$ has a nonparabolic cut point $c$ with orbit $C=\Gamma (c)$.
Let $D=D(M,C)$ be the dendrite quotient of $M$ given by Proposition~\ref{prop: reldendrite}.
If $R$ is the topological $\R$--tree obtained from $D$ by removing all terminal points, then each interval of $R$ is stabilized by a finite subgroup of $\Gamma$.
In particular, an infinite subgroup of $\Gamma$ cannot fix more than one point of $R$.

Consider the simplicial $(\mathbb{L},\P)$--tree $T_\ell$ given by Proposition~\ref{prop:Accessibility}. For each vertex $v$ of $T_\ell$, the stabilizer $\Gamma_v$ has a fixed point $p_v$ in $R$.  Since $\Gamma_v$ is infinite, its fixed point $p_v$ must be unique.
For any edge $e$ of $T_\ell$ joining vertices $w$ and $w$, the group $\Gamma_e$ stabilizes the arc $[p_v,p_w]$ in $R$.
Since $\Gamma_e$ is also infinite, it follows that the points $p_v$ and $p_w$ fixed by $\Gamma_v$ and $\Gamma_w$ must be equal, so all vertex groups have the same unique fixed point $p$ in $R$.
Since $p = \Fix(\Gamma_{g(v)}) = \Fix( g \Gamma_v g^{-1}) = g(p)$ for any $g \in \Gamma$ and any vertex $v$, the action on $R$ leaves $p$ fixed.
This conclusion contradicts Proposition~\ref{prop: reldendrite}, which states that $\Gamma$ does not have a fixed point in the dendrite $D$.
\end{proof}

Recall that a \emph{Peano continuum} is a compact, connected, locally connected, metrizable space.
Extending a result of Bestvina--Mess \cite{BestvinaMess91}, Bowditch shows in \cite[\S 9]{Bowditch01_Peripheral} that a relatively hyperbolic group pair satisfying the conclusion of  Proposition~\ref{prop:NoSplitCutPoint} has locally connected boundary. 
We provide details for the benefit of the reader.

\begin{proposition}
\label{prop:PeanoNoCutPoint}
Let $(\Gamma,\P)$ be relatively hyperbolic.
Suppose $\Gamma$ does not split over a finite or parabolic subgroup relative to $\P$.
Then $M=\boundary(\Gamma,\P)$ is a Peano continuum without cut points.
\end{proposition}

\begin{proof}
Since $\Gamma$ is finitely generated, $M$ is the Gromov boundary of a certain $\delta$--hyperbolic space $X$ that has constant horospherical distortion in the sense of Bowditch \cite{Bowditch99_Boundaries}.
In particular, $M$ is metrizable.

Suppose $M$ is not locally connected. 
Since $M$ is connected, \cite[Thm.~1.1]{Bowditch99_Boundaries} implies that the space $X$ has a separating horoball based at a point $\eta\in M$. 
The base point of any separating horoball in $X$ is a cut point of $M=\boundary X$ (see \cite[Prop.~3.4]{HruskaRuane_Hyperbolic}).
By Proposition~\ref{prop:NoSplitCutPoint}, the cut point $\eta$ is a parabolic point of $M$.
However, since $\Gamma$ does not split over a parabolic subgroup relative to $\P$, the space $X$ cannot have a separating horoball based at a parabolic point \cite[Thm.~1.2]{Bowditch99_Boundaries}.
We conclude that $M$ is locally connected, so that $M$ is a Peano continuum.

Finally, if the Peano continuum $M$ had a cut point $\eta$, then by
Proposition~\ref{prop:NoSplitCutPoint} 
such a point would be parabolic. By \cite[Prop.~5.1]{Bowditch99_Boundaries} the space $X$ would contain a separating horoball based at $\eta$, which is impossible.
\end{proof}

The proof of Theorem~\ref{thm: intro main} combines Proposition~\ref{prop:PeanoNoCutPoint} and Theorem~\ref{thm: localconnectednessInfiniteCase}.

\begin{proof}[Proof of Theorem~\ref{thm: intro main}]
One-endedness implies that $M$ is connected.
Consider the JSJ tree of cylinders $T_{JSJ}$ for $\Gamma$ over parabolic subgroups relative to $\P$ given by Proposition~\ref{prop:JSJfinitelygenerated}.
For each vertex $v \in V_0(T_{JSJ})$, the stabilizer $\Gamma_v$ is hyperbolic relative to $\Q_v$ and does not split relative to $\P$ over any finite or parabolic subgroup.
By Proposition~\ref{prop:PeanoNoCutPoint}, the boundary $\boundary(\Gamma_v,\Q_v)$ is a Peano continuum with no cut points.
Thus, by Theorem~\ref{thm: localconnectednessInfiniteCase}, the boundary $M$ is locally connected.
We also conclude that $M$ has a cut point if and only if the tree $T_{JSJ}$ is nontrivial, which happens if and only if $\Gamma$ splits relative to $\P$ over a parabolic subgroup.

Recall that $M$ is the limit of the tree system $\Theta$ associated to $T_{JSJ}$ by Proposition~\ref{prop:TreeOfSpaces}.
We show that a nonparabolic point of $M$ cannot be a cut point, using the notation from Section~\ref{sec:BoundaryTreeSystem}.
For each $v \in V_0(T_{JSJ})$, the boundary $M$ maps onto $M_v = \boundary(\Gamma_v,\Q_v)$ by a retraction that collapses to a point each halfspace $H(\vec{e}\,)$ not containing $M_v$.
It follows that any nonparabolic cut point of $M$ contained in $M_v$ would also be a cut point of $M_v$. Thus no such cut point exists.
The subspace $\bigcup_v M_v$ is a dense connected subspace of $M$ since it contains the vertex set $V(K)$, where $K$ is any fine hyperbolic graph $K$ with a relatively hyperbolic action compatible with $T_{JSJ}$ in the sense of Definition~\ref{defn:Compatible}. If $\xi$ is in its complement, then
\[
   {\textstyle \bigcup_v M_v \subset M\setminus\{\xi\} \subset \overline{\bigcup_v M_v}=M.}
\]
Thus $M-\{\xi\}$ is connected and $\xi$ is also not a cut point of $M$.
\end{proof}

We conclude by observing that the proof of Theorem~\ref{thm: intro main} also establishes Theorem~\ref{thm: intro cut point}.

\bibliographystyle{alpha}
\bibliography{relative}

\begin{thebibliography}{Bow99d}

\bibitem[AS85]{AncelSiebenmann85}
F.D. Ancel and L.~Siebenmann.
\newblock The construction of homogeneous homology manifolds.
\newblock {\em Abstracts Amer. Math. Soc.}, 6, 1985.
\newblock Abstract number 816-57-72.

\bibitem[BC01]{BowditchCrisp01}
B.H. Bowditch and J.~Crisp.
\newblock Archimedean actions on median pretrees.
\newblock {\em Math. Proc. Cambridge Philos. Soc.}, 130(3):383--400, 2001.

\bibitem[BF95]{BestvinaFeighn_RipsMachine}
M.~Bestvina and M.~Feighn.
\newblock Stable actions of groups on real trees.
\newblock {\em Invent. Math.}, 121(2):287--321, 1995.

\bibitem[BK05]{BonkKleiner05}
M.~Bonk and B.~Kleiner.
\newblock Quasi-hyperbolic planes in hyperbolic groups.
\newblock {\em Proc. Amer. Math. Soc.}, 133(9):2491--2494, 2005.

\bibitem[BM91]{BestvinaMess91}
M.~Bestvina and G.~Mess.
\newblock The boundary of negatively curved groups.
\newblock {\em J. Amer. Math. Soc.}, 4(3):469--481, 1991.

\bibitem[Bow98]{Bowditch98_JSJ}
B.H. Bowditch.
\newblock Cut points and canonical splittings of hyperbolic groups.
\newblock {\em Acta Math.}, 180(2):145--186, 1998.

\bibitem[Bow99a]{Bowditch99_Boundaries}
B.H. Bowditch.
\newblock Boundaries of geometrically finite groups.
\newblock {\em Math. Z.}, 230(3):509--527, 1999.

\bibitem[Bow99b]{Bowditch99_Connectedness}
B.H. Bowditch.
\newblock Connectedness properties of limit sets.
\newblock {\em Trans. Amer. Math. Soc.}, 351(9):3673--3686, 1999.

\bibitem[Bow99c]{Bowditch99_Convergence}
B.H. Bowditch.
\newblock Convergence groups and configuration spaces.
\newblock In J.~Cossey, C.F. Miller, W.D. Neumann, and M.~Shapiro, editors,
  {\em Geometric group theory down under \textup{(}Canberra, 1996\textup{)}},
  pages 23--54. de Gruyter, Berlin, 1999.

\bibitem[Bow99d]{Bowditch99_Treelike}
B.H. Bowditch.
\newblock Treelike structures arising from continua and convergence groups.
\newblock {\em Mem.\ Amer.\ Math.\ Soc.}, 139(662):1--86, 1999.

\bibitem[Bow01]{Bowditch01_Peripheral}
B.H. Bowditch.
\newblock Peripheral splittings of groups.
\newblock {\em Trans. Amer. Math. Soc.}, 353(10):4057--4082, 2001.

\bibitem[Bow12]{Bowditch12_RelHyp}
B.H. Bowditch.
\newblock Relatively hyperbolic groups.
\newblock {\em Internat. J. Algebra Comput.}, 22(3):1250016, 66, 2012.

\bibitem[BW13]{BigdelyWise13}
H.~Bigdely and D.T. Wise.
\newblock Quasiconvexity and relatively hyperbolic groups that split.
\newblock {\em Michigan Math. J.}, 62(2):387--406, 2013.

\bibitem[Cap54]{Capel54}
C.E. Capel.
\newblock Inverse limit spaces.
\newblock {\em Duke Math. J.}, 21:233--245, 1954.

\bibitem[Dah03]{Dahmani03}
F.~Dahmani.
\newblock Combination of convergence groups.
\newblock {\em Geom. Topol.}, 7:933--963, 2003.

\bibitem[DS05]{DrutuSapir05}
C.~Dru{\cb{t}}u and M.~Sapir.
\newblock Tree-graded spaces and asymptotic cones of groups.
\newblock {\em Topology}, 44(5):959--1058, 2005.
\newblock With an appendix by D. Osin and M. Sapir.

\bibitem[DT17]{DowdallTaylor17}
S.~Dowdall and S.J. Taylor.
\newblock The co-surface graph and the geometry of hyperbolic free group
  extensions.
\newblock {\em J. Topol.}, 10(2):447--482, 2017.

\bibitem[Eng89]{Engelking_Topology}
R.~Engelking.
\newblock {\em General topology}, volume~6 of {\em Sigma Series in Pure
  Mathematics}.
\newblock Heldermann Verlag, Berlin, 2nd edition, 1989.

\bibitem[Ger09]{Gerasimov09}
V.~Gerasimov.
\newblock Expansive convergence groups are relatively hyperbolic.
\newblock {\em Geom. Funct. Anal.}, 19(1):137--169, 2009.

\bibitem[Ger12]{Gerasimov12}
V.~Gerasimov.
\newblock Floyd maps for relatively hyperbolic groups.
\newblock {\em Geom. Funct. Anal.}, 22(5):1361--1399, 2012.

\bibitem[GL07]{GuirardelLevitt_Deformation}
V.~Guirardel and G.~Levitt.
\newblock Deformation spaces of trees.
\newblock {\em Groups Geom. Dyn.}, 1(2):135--181, 2007.

\bibitem[GL15]{GuirardelLevitt_Splittings}
V.~Guirardel and G.~Levitt.
\newblock Splittings and automorphisms of relatively hyperbolic groups.
\newblock {\em Groups Geom. Dyn.}, 9(2):599--663, 2015.

\bibitem[GL17]{GuirardelLevitt_JSJ}
V.~Guirardel and G.~Levitt.
\newblock {\em J{SJ} decompositions of groups}, volume 395 of {\em
  Ast\'erisque}.
\newblock Soci\'et\'e Math\'ematique de France, Paris, 2017.

\bibitem[GM08]{GrovesManning08}
D.~Groves and J.F. Manning.
\newblock Dehn filling in relatively hyperbolic groups.
\newblock {\em Israel J. Math.}, 168:317--429, 2008.

\bibitem[GP15]{GerasimovPotyagailo15}
V.~Gerasimov and L.~Potyagailo.
\newblock Non--finitely generated relatively hyperbolic groups and {F}loyd
  quasiconvexity.
\newblock {\em Groups Geom. Dyn.}, 9(2):369--434, 2015.

\bibitem[GP16]{GerasimovPotyagailo16}
V.~Gerasimov and L.~Potyagailo.
\newblock Similar relatively hyperbolic actions of a group.
\newblock {\em Int. Math. Res. Not. IMRN}, 2016(7):2068--2103, 2016.

\bibitem[GS19]{GeogheganSwenson19}
R.~Geoghegan and E.~Swenson.
\newblock On semistability of {$\CAT(0)$} groups.
\newblock {\em Groups Geom. Dyn.}, 13(2):695--705, 2019.

\bibitem[Hau19]{Haulmark19}
M.~Haulmark.
\newblock Local cut points and splittings of relatively hyperbolic groups.
\newblock {\em Algebr. Geom. Topol.}, 19(6):2795--2836, 2019.

\bibitem[HH23]{HaulmarkHruska_Canonical}
M.~Haulmark and G.C. Hruska.
\newblock On canonical splittings of relatively hyperbolic groups.
\newblock {\em Israel J. Math.}, 258(1):249--286, 2023.

\bibitem[HM22]{HaulmarkMihalik_RelativeSemistable}
M.~Haulmark and M.~Mihalik.
\newblock Relatively hyperbolic groups with semistable peripheral subgroups.
\newblock {\em Internat. J. Algebra Comput.}, 32(4):753--783, 2022.

\bibitem[HR]{HruskaRuane_Hyperbolic}
G.C. Hruska and K.~Ruane.
\newblock Hyperbolic groups and local connectivity.
\newblock To appear in R. Geoghegan, C.~Guilbault, and K.~Ruane, editors,
  \textit{Topology at infinity of discrete groups}, a volume of
  \textit{Contemporary Mathematics}. Amer.\ Math.\ Soc. arXiv:2308.14964.

\bibitem[Hru10]{Hruska10}
G.C. Hruska.
\newblock Relative hyperbolicity and relative quasiconvexity for countable
  groups.
\newblock {\em Algebr. Geom. Topol.}, 10(3):1807--1856, 2010.

\bibitem[HW23]{HruskaWalsh_Planar}
G.C. Hruska and G.S. Walsh.
\newblock Planar boundaries and parabolic subgroups.
\newblock {\em Math. Res. Lett.}, 30(4):1081--1112, 2023.

\bibitem[Jak91]{Jakobsche91}
W.~Jakobsche.
\newblock Homogeneous cohomology manifolds which are inverse limits.
\newblock {\em Fund. Math.}, 137(2):81--95, 1991.

\bibitem[Kec95]{Kechris95}
A.S. Kechris.
\newblock {\em Classical descriptive set theory}, volume 156 of {\em Graduate
  Texts in Mathematics}.
\newblock Springer-Verlag, New York, 1995.

\bibitem[KK00]{KapovichKleinier00}
M.~Kapovich and B.~Kleiner.
\newblock Hyperbolic groups with low-dimensional boundary.
\newblock {\em Ann. Sci. \'{E}cole Norm. Sup. \textup{(}4\textup{)}},
  33(5):647--669, 2000.

\bibitem[KR14]{KapovichRafi14}
I.~Kapovich and K.~Rafi.
\newblock On hyperbolicity of free splitting and free factor complexes.
\newblock {\em Groups Geom. Dyn.}, 8(2):391--414, 2014.

\bibitem[Lev98]{Levitt98}
G.~Levitt.
\newblock Non-nesting actions on real trees.
\newblock {\em Bull. London Math. Soc.}, 30(1):46--54, 1998.

\bibitem[McD06]{McDuff06}
D.~McDuff.
\newblock Groupoids, branched manifolds and multisections.
\newblock {\em J. Symplectic Geom.}, 4(3):259--315, 2006.

\bibitem[ML98]{MacLane_Categories}
S.~Mac~Lane.
\newblock {\em Categories for the working mathematician}, volume~5 of {\em
  Graduate Texts in Mathematics}.
\newblock Springer-Verlag, New York, second edition, 1998.

\bibitem[MS84]{MorganShalen84}
J.W. Morgan and P.B. Shalen.
\newblock Valuations, trees, and degenerations of hyperbolic structures. {I}.
\newblock {\em Ann. of Math. \textup{(}2\textup{)}}, 120(3):401--476, 1984.

\bibitem[MS21]{MihalikSwenson21}
M.~Mihalik and E.~Swenson.
\newblock Relatively hyperbolic groups with semistable fundamental group at
  infinity.
\newblock {\em J. Topol.}, 14(1):39--61, 2021.

\bibitem[Osi06a]{Osin06_Elementary}
D.V. Osin.
\newblock Elementary subgroups of relatively hyperbolic groups and bounded
  generation.
\newblock {\em Internat. J. Algebra Comput.}, 16(1):99--118, 2006.

\bibitem[Osi06b]{Osin06_RelHyp}
D.V. Osin.
\newblock Relatively hyperbolic groups: {I}ntrinsic geometry, algebraic
  properties, and algorithmic problems.
\newblock {\em Mem. Amer. Math. Soc.}, 179(843):1--100, 2006.

\bibitem[PS06]{PapasogluSwenson06}
P.~Papasoglu and E.~Swenson.
\newblock From continua to {$\mathbb{R}$}--trees.
\newblock {\em Algebr. Geom. Topol.}, 6:1759--1784, 2006.

\bibitem[Ser77]{Serre_Trees}
J.-P. Serre.
\newblock {\em Arbres, amalgames, {${\rm SL}\sb{2}$}}, volume~46 of {\em
  Ast\'erisque}.
\newblock Soci\'et\'e Math\'ematique de France, Paris, 1977.
\newblock Written in collaboration with Hyman Bass.

\bibitem[Sou]{deSouza_Blowups}
L.H.R.~de Souza.
\newblock Equivariant blowups of bounded parabolic points.
\newblock arXiv:2008.05822.

\bibitem[Swa96]{Swarup96}
G.A. Swarup.
\newblock On the cut point conjecture.
\newblock {\em Electron. Res. Announc. Amer. Math. Soc.}, 2(2):98--100, 1996.

\bibitem[Swe00]{Swenson00}
E.L. Swenson.
\newblock A cutpoint tree for a continuum.
\newblock In M.~Atkinson, N.~Gilbert, J.~Howie, S.~Linton, and E.~Robertson,
  editors, {\em Computational and geometric aspects of modern algebra}, volume
  275 of {\em London Math. Soc. Lecture Note Ser.}, pages 254--265. Cambridge
  Univ. Press, Cambridge, 2000.

\bibitem[{\'{S}}wi20]{Swiatkowski20}
J.~{\'{S}}wi{\k{a}}tkowski.
\newblock Trees of metric compacta and trees of manifolds.
\newblock {\em Geom. Topol.}, 24(2):533--592, 2020.

\bibitem[Tuk94]{Tukia94}
P.~Tukia.
\newblock Convergence groups and {G}romov's metric hyperbolic spaces.
\newblock {\em New Zealand J. Math.}, 23(2):157--187, 1994.

\bibitem[Tuk96]{Tukia94Erratum}
P.~Tukia.
\newblock Erratum: ``{C}onvergence groups and {G}romov's metric hyperbolic
  spaces''.
\newblock {\em New Zealand J. Math.}, 25(1):105--106, 1996.

\bibitem[Wil70]{Willard_Topology}
S.~Willard.
\newblock {\em General topology}.
\newblock Addison-Wesley Publishing Co., Reading, Mass.-London-Don Mills, Ont.,
  1970.

\end{thebibliography}

\end{document}